\documentclass[11pt]{amsart}
\usepackage{amssymb,amsmath,amsfonts,latexsym, xcolor}

\usepackage[all,cmtip]{xy}
\newtheorem{thm}{Theorem}[section]
\newtheorem{prop}[thm]{Proposition}

\newtheorem{cor}[thm]{Corollary}

\newtheorem{definition}[thm]{Definition}
\newtheorem{example}[thm]{Example}

\newtheorem{remark}[thm]{Remark}
\newtheorem{lemma}[thm]{Lemma}
\newtheorem{theorem}[thm]{Theorem}
\numberwithin{equation}{section}

\newcommand{\T}{\mathbb{T}}
\newcommand{\C}{\mathbb{C}}

\newcommand{\Z}{\mathbb{Z}}

\newcommand{\clb}{\mathcal{B}}

\newcommand{\cld}{\mathcal{D}}
\newcommand{\cle}{\mathcal{E}}
\newcommand{\clf}{\mathcal{F}}

\newcommand{\clh}{\mathcal{H}}
\newcommand{\clk}{\mathcal{K}}

\newcommand{\clm}{\mathcal{M}}

\newcommand{\clq}{\mathcal{Q}}

\newcommand{\cls}{\mathcal{S}}

\newcommand{\raro}{\rightarrow}

\newcommand{\norm}[1]{\left\Vert#1\right\Vert}

\begin{document}
		
\title[Representations of odometer semigroup]{Contractive representations of odometer semigroup}

\author[Ghatak]{Anindya Ghatak}
\address{Anindya Ghatak, Department of Mathematics - Noida - Bennett University, TechZone 2, Greater Noida, Uttar Pradesh 201310, India}
\email{anindya.ghatak@bennett.edu.in}

\author[Rakshit]{Narayan Rakshit}
\address{Narayan Rakshit, Department of Mathematics, Indian Institute of Technology, Roorkee, Uttarakhand 247667, India}
\email{nrakshit@ma.iitr.ac.in}

\author[Sarkar]{Jaydeb Sarkar}
\address{Jaydeb Sarkar, Indian Statistical Institute, Statistics and Mathematics Unit, 8th Mile, Mysore Road, Bangalore, 560059,
India}
\email{jay@isibang.ac.in, jaydeb@gmail.com}

\author[Suryawanshi]{Mansi Suryawanshi}
\address{Mansi Suryawanshi, Statistics and Mathematics Unit, Indian Statistical Institute, 8th Mile, Mysore Road, Bangalore, Karnataka - 560059, India}
\email{mansisuryawanshi1@gmail.com}

\begin{abstract}
Given a natural number $n \geq 1$, the odometer semigroup $O_n$, also known as the adding machine or the Baumslag-Solitar monoid with two generators, is a well-known object in group theory. This paper examines the odometer semigroup in relation to representations of bounded linear operators. We focus on noncommutative operators and prove that contractive representations of $O_n$ always admit to nicer representations of $O_n$. We give a complete description of representations of $O_n$ on the Fock space and relate it to the odometer lifting and subrepresentations of $O_n$. Along the way, we also classify Nica covariant representations of $O_n$.
\end{abstract}


\subjclass[2020]{47A13, 46L05, 47A20, 30H10, 20E08, 32A35}

\keywords{Odometer semigroup, row contractions, invariant subspaces, isometric representations, dilations, Nica covariant representations, subrepresentations}

\maketitle

\tableofcontents

\section{Introduction}

Let $n \geq 1$ be a natural number. The \textit{odometer semigroup} $O_n$ is a semigroup formed by $n$ generators $\{v_1, \ldots, v_n\}$ along with a generator $w$ that satisfies the conditions
\[
w v_i =
\begin{cases}
v_{i+1} & \mbox{if } i = 1, \ldots, n-1
\\
v_1 w & \mbox{if } i = n.
\end{cases}
\]
This semigroup is frequently referred to as the adding machine or the Baumslag-Solitar monoid $BS(1,n)^+$ with two generators \( a \) and \( b \) satisfying the relation \( b^n a = ab \).
This can also be seen as a Zappa-Sz\'{e}p product of the free semigroup with $n$ generators \cite{GT}. Researchers working in group theory have widely recognized the importance of this object; for instance, see \cite{Brin, Delgado} and the references therein for more advancement. This is also useful for understanding algebraic-analytic structures, including Toeplitz algebras \cite{Clark}, von Neumann algebras \cite{Fima}, $C^*$-algebras \cite{Spiel}, semigroup $C^*$-algebras \cite{Brownlowe, Li Yang, Li, Nica}, and numerous others. Our main goal in this paper is to enhance it with representations of bounded linear operators on Hilbert spaces.

Let $\clh$ be a Hilbert space. In this paper, all Hilbert spaces are separable and over $\mathbb{C}$. Let $\clh^n$ denote the Hilbert space of the $n$-direct sum of $\clh$ with itself. By an \textit{$n$-row operator}, or simply a \textit{row operator} (as $n$ will be clear from the context), we mean a bounded linear operator
\[
T: = (T_1, \ldots, T_n): \clh^n \raro \clh,
\]
where $T_i \in \clb(\clh)$ for all $i=1, \ldots, n$. Here, given Hilbert spaces $\clh$ and $\clk$, we set the space of bounded linear operators from $\clh$ into $\clk$ by $\clb(\clh, \clk)$. If $\clh = \clk$, we simply write this space as $\clb(\clh)$. Given a row operator $T$ acting on $\clh$ as above and $W \in \clb(\clh)$, we say that $(W, T)$ is a \textit{representation of $O_n$} if the following conditions are satisfied:
\begin{equation} \label{eqn def OS}
WT_i = \begin{cases}
T_{i+1} & \mbox{if } i=1, \ldots, n-1
\\
T_1 W & \mbox{if } i=n.
\end{cases}
\end{equation}
Without a doubt, the general representations of $O_n$ are a too-wide problem that might not yield much unless certain meaningful restrictions are put on the row operators. This is where row contractions enter into our theory. A \textit{row contraction} on $\clh$ is defined as a row operator $T = (T_1, \ldots, T_n)$ that meets the following condition
\[
\sum_{i=1}^n T_i T_i^* \leq I.
\]
Equivalently, the row operator $T \in \clb(\clh^n, \clh)$ is a contraction. The row contraction $T$ is said to be \textit{pure} if
\begin{equation}\label{eqn pure}
\lim_{m \raro \infty} \sum_{|\mu| = m, \mu \in F_n^+} \|T^*_\mu f\|^2 = 0,
\end{equation}
for all $f \in \clh$. Here $F_n^+$ denotes the free semigroup generated by $n$ letters $\{g_1, \ldots, g_n\}$ with identity $g_0$, and for each word $\mu = g_{\mu_1} \cdots g_{\mu_k} \in F_n^+$, $k \geq 1$, we write $T_\mu = T_{\mu_1} \cdots T_{\mu_k}$, and $T_{g_0} = I_\clh$. The above limit condition is well-accepted and has broad applicability that aligns seamlessly with our context, which we shall elaborate on shortly. In general, we remark that the theory of noncommutative pure row contractions is renowned for its abundance, creativity, and ability to integrate a number of established theories  \cite{Popescu 95, Popescue char, Popescue}, such as the Sz.-Nagy and Foias model theory, interpolation problem, Toeplitz operators, and $C^*$-algebras. The effect of the noncommutative dilation theory will also be heavily reflected in this paper. At this point, we introduce the central notion of this investigation.

\begin{definition}\label{def cont rep}
A contractive representation of $O_n$ is a representation $(W, T)$ of $O_n$ such that $T$ is a pure contraction.
\end{definition}

It is worth noting that the notion of a contractive representation, as introduced above, can also be defined for row contractions without assuming purity. However, in this paper, we primarily focus on pure row contractions. For this reason, we impose the purity condition on the row operator $T$.

The goal of this paper is to present a complete picture of representations of $O_n$. Along the way, we offer a precise depiction of the operator $W$, the other generator of representations of $O_n$. The latter direction enhances the existing noncommutative dilation theory, which also relates to joint invariant subspaces (or subrepresentations) of certain natural contractive representations of $O_n$. Before we get into more details about the main results of this paper, we need to set up a notion that also provides natural examples of contractive representations of $O_n$.

A row contraction $T$ is called \textit{row isometry} if $T_{i}$ is isometry for all $i=1, \ldots, n$. Note that a row contraction $T$ is a row isometry if and only if $T^*_{i}T_{j}= \delta_{ij} I$ for all $i, j = 1, \ldots, n$. Creation operators on the (full) Fock space provides concrete examples of row isometries: Denote by $\mathcal{F}^{2}_n$ the \textit{Fock space} over $\mathbb{C}^{n}$, that is
\[
\mathcal{F}^{2}_n=  \mathbb{C}\Omega\oplus \bigoplus\limits_{k=1}^{\infty} ({\mathbb{C}^n})^{\otimes k},
\]
where $\Omega$ is a unit vector called \emph{vacuum state}. Let  $\{e_{1},\ldots, e_{n}\}$ be the standard orthonormal basis of $\mathbb{C}^{n}$. For each $i = 1, \ldots, n$, the $i$-th \textit{creation} or \textit{left shift operator} $S_i$ on $\clf^2_n$ is defined by
\[
S_i f = e_i \otimes f \qquad (f \in \clf^2_n).
\]
A simple computation shows that $(S_1, \ldots, S_n)$ is a pure row isometry. More generally, if $\cle$ is a Hilbert space, then
\[
S^\cle :=(S_1 \otimes I_\cle, \ldots, S_n \otimes I_\cle),
\]
on $\clf^2_n \otimes \cle$ is a pure row isometry \cite[Remark 1.1]{Popescue}. We furthermore emphasize that the creation operators on vector-valued Fock spaces are highly significant in noncommutative operator theory and are utilized in free analytic models \cite{Popescue char}. For instance, the noncommutative Wold type decomposition theorem  (see \cite[Theorem 2]{Frazho 84} and \cite[Theorem 1.3]{Popescue}) says that up to unitary equivalence, a pure row isometry is of the form $S^\cle$ for some Hilbert space $\cle$.

Our theory will center on representations of $O_n$ that correspond to row isometries on vector-valued Fock spaces. We introduce them formally as follows:

\begin{definition}\label{def Fock rep}
Let $\cle$ be a Hilbert space and let $(W, S^\cle)$ be a representation of $O_n$ for some $W \in \clb(\clf^2_n \otimes \cle)$. We refer to $(W, S^\cle)$ as a Fock representation of $O_n$, or simply a Fock representation. If, in addition, $W$ is an isometry (unitary/contractive) operator, then we call it an isometric (unitary/contractive) Fock representation.
\end{definition}

We need to fix a notation. For each word $\mu \in F_n^+$, we write
\[
e_{\mu} =
\begin{cases}
e_{\mu_{1}}\otimes \cdots \otimes e_{\mu_k} & \mbox{if } \mu = g_{\mu_{1}}g_{\mu_{2}}\hdots g_{\mu_k}
\\
\Omega & \mbox{if }\mu = g_{0}.
\end{cases}
\]
It follows that the set $\{e_\mu: \mu \in F_n^+\}$ is an orthonormal basis for $\clf_n^2$. As already pointed out, Fock representations will play an important role in our analysis. Our first result yields a complete description as well as a classification of Fock representations (see Theorem \ref{Prop: Odometer operator}).

\begin{theorem}\label{thm intro odometer maps}
Let $\cle$ be a Hilbert space, and let $W \in \clb(\clf^2_n \otimes \cle)$. Then $(W, S^{\cle})$ is a Fock representation if and only if there exists $L \in \clb(\cle,  \clf^2_n \otimes \cle)$ such that
\[
W = W_L,
\]
where $W_L \in \clb(\clf^2_n \otimes \cle)$ is defined by
\[
W_{L}(\Omega \otimes \eta) = L\eta,
\]
for all $\eta \in \cle$, and for all $e_\mu = e_{{\mu_{1}}} \otimes\cdots \otimes e_{\mu_{m}} \in F^+_n$, $\mu \neq g_0$, define
\[
W_{L}(e_{\mu} \otimes \eta) =
\begin{cases}
e_{\mu_1+1} \otimes e_{\mu_2}\otimes\cdots \otimes e_{\mu_{m}} \otimes \eta & \mbox{if } \mu_{1} \neq n
\\
e_{{1}} \otimes e_{{\mu_{2}+1}} \otimes e_{\mu_3} \otimes \cdots \otimes e_{\mu_{m}} \otimes \eta & \mbox{if } \mu_{1}=n, \mu_{2} \neq n
\\
e^{\otimes 2}_{1}   \otimes e_{{\mu_{3}+1}} \otimes e_{\mu_4} \otimes \cdots \otimes e_{\mu_{m}} \otimes \eta & \mbox{if } \mu_{1}=\mu_{2}=n, \mu_{3} \neq n
\\
\vdots & \vdots
\\
e^{\otimes m}_{1} \otimes L \eta & \mbox{if } \mu_{1}=\cdots= \mu_{m}=n.
\end{cases}
\]
Moreover, if $W = W_L$ for some $L \in \clb(\cle,  \clf^2_n \otimes \cle)$, then $L$ is unique.
\end{theorem}

Such a map $W_L$ is referred to as an \textit{odometer map} with the symbol $L \in \clb(\cle,  \clf^2_n \otimes \cle)$. It is a very curious fact that the odometer maps are defined for all $L \in \clb(\cle,  \clf^2_n \otimes \cle)$, and subsequently, the theme of symbols becomes comparable to the noncommutative Toeplitz operators \cite{Popescu 95} (see Remark \ref{remark odometer maps} for more on this).

A special form of the above operator $W_L$ was introduced by Li \cite{Boyu Li22} in a restricted context, namely when $W_L$ is unitary. He proved, in the language at hand, that $(W, S^\cle)$ is a unitary Fock representation if and only if there exists $L \in \clb(\cle)$ such that $W = W_L$ \cite[Corollary 3.6]{Boyu Li22}. In our general case, we needed to adopt a more finer notion of the model operator $W$ acting on vector-valued Fock spaces. Ultimately, we effectively improved Li's notion, which amounts to extending the idea of symbols from $\clb(\cle)$ to $\clb(\cle, \clf^2_n \otimes \cle)$. We have applied this to the definition of $W_L$ in the above theorem. This paper will also retrieve Li's result using a completely different methodology (see Theorem \ref{W_L uni}). In Remark \ref{remark odometer maps}, we will discuss more about the odometer maps and corresponding symbols.

In Theorem \ref{ W_L iso iff L}, we present a complete classification of isometric Fock representations. To explain the result, for each $L\in \mathcal{B} (\cle, \clf^2_n \otimes \cle)$, we define the closed subspace (see \eqref{eqn E_L})
\[
\cle_L = \overline{\rm span}\{ {e^{\otimes{m}}_{1} \otimes {\eta}} : m \in \mathbb{Z}_{+}, \eta \in \cle\} \ominus \overline{\rm span}\{ {e^{\otimes{p}}_{1} \otimes {L \zeta}} : p \geq 1, \zeta \in \cle\}.
\]
We shall also designate $W_L$ as the odometer map that corresponds to the symbol $L$, as defined in Theorem \ref{thm intro odometer maps}. We have the following classification of isometric Fock representations:

\begin{thm}
Let $\cle$ be a Hilbert space, and let $L \in \clb(\cle, \clf^2_n \otimes \cle)$. Then $(W_L, S^{\cle})$ is an isometric Fock representation if and only if the following conditions hold:
\begin{enumerate}
\item $L$ is an isometry.
\item $L \cle \subseteq \cle_L$.
\end{enumerate}
\end{thm}

In the scalar case (that is, $\cle = \mathbb{C}$), odometer maps are precisely given by $W_\xi$, where $\xi$ is in $\clf^2_n$. In this case, the above result becomes even more concrete (see Corollary \ref{cor scalar iso Fock rep}): Let $\xi \in \clf^2_n$. Then $W_{\xi}$ on $\clf^2_n$ is an isometry if and only if
\[
\xi = \sum^{\infty}_{p=0} c_{p}e^{\otimes p}_{1},
\]
where the sequence of scalars $\{c_p\}_{p=0}^\infty$ satisfies the conditions $\sum^{\infty}_{p=0} \vert c_{p}\vert^2=1$, and $\sum^{\infty}_{p=0}  c_{p+r} \overline{c_{p}}=0$ for all $r\geq 1$.

The existence of vectors in $\clf^2_n$ that satisfy the above conditions may appear counterintuitive. In this regard, we will present numerous constructions of such $\xi$ in $\clf^2_n$ in the final section. There, we will also present several examples of general isometric Fock representations.

Clearly, the evaluation of the conjugates of odometer maps is a basic question. How nice the general formula could be is not apparent. However, the conjugate operator formula is precise and significant for isometric odometer maps (see Proposition \ref{Lemma: Adjoint of W}). This is a technique we applied to understand Nica-covariant representations. In fact, this is a specific category of Fock representations of $O_n$ that is of interest (see \cite[page 742]{Boyu Li22} and also see \cite{Brownlowe, Li, Nica}):

\begin{definition}\label{def Nica}
Let $\cle$ be a Hilbert space, and let $(W, S^\cle)$ be an isometric Fock representation of $O_n$. We say that $(W, S^\cle)$ is a Nica-covariant representation of $O_n$ if
\[
W^*(S_1 \otimes I_\cle) = (S_n \otimes I_\cle) W^*.
\]
\end{definition}

The classification of Nica covariant representations of $O_n$ is rather clean (see Theorem \ref{Thm: Nica covariant characterization}):
\begin{theorem}
Let $(W_L, S^\cle)$ be an isometric Fock representation. Then $(W_L, S^\cle)$ is a Nica-covariant representation of $O_n$ if and only if
\[
L \cle \subseteq \Omega \otimes \cle.
\]
\end{theorem}

We will refer to a symbol that satisfies the above criteria as a \textit{constant symbol}. Simply put, constant symbols precisely implement Nica-covariant representations of $O_n$.

As a result of the classification of general Nica covariant representations of $O_n$ along with some computations, we recover the structure of unitary Fock representations that were previously obtained in \cite[ Corollary 3.6]{Boyu Li22}. Pursuant to the terminology used in our present paper, it states (see Theorem \ref{W_L uni}):

\begin{theorem}
Let $\cle$ be a Hilbert space, and let $L \in \clb(\cle, \clf^2_n \otimes \cle)$ be an isometry. Then $W_L$ is unitary if and only if
\[
L \cle = \Omega \otimes \cle.
\]
\end{theorem}

When $\cle$ is finite dimensional, we get the following characterization for the Nica-covariant representations of $O_n$(see Theorem \ref{Thm: scalar Nica covariant characterization}):

\begin{theorem}
Let $\mathcal{E}$  be a finite-dimensional Hilbert space and let $L\in \mathcal{B}(\mathcal{E},\mathcal{F}^2_n\otimes \mathcal{E})$ be such that $(W_L,S^{\mathcal{E}})$ becomes an isometric representation of $O_n$. Then, the following are equivalent:
\begin{enumerate}
\item $(W_L,S^{\mathcal{E}})$ is a Nica-covariant representation of $O_n$.
\item $L:\mathcal{E}\to \Omega \otimes\mathcal{E}$ is a unitary.
\item $W_L$ is a unitary.
\end{enumerate}
\end{theorem}

In the scalar case, the above further simplifies to the following (see Corollary \ref{scalar}): Let $\xi \in \clf^2_n$, and suppose $(W_{\xi}, S)$ is an isometric Fock representation of $O_n$ on $\clf^2_n$. Then the following are equivalent:
\begin{enumerate}
\item $(W_{\xi}, S)$ is a Nica covariant representation of $O_n$.
\item There exists $c \in \T$ such that $\xi= c \Omega$.
\item $W_{\xi}$ is unitary.
\end{enumerate}

Now we turn to the noncommutative dilation theory, and connect it with representations of $O_n$. Two row operators $T=(T_1, \ldots, T_n)$ on $\clh$ and $R = (R_1, \ldots, R_n)$ on $\clk$ are considered \textit{unitarily equivalent} if there exists a unitary $U \in \clb(\clh, \clk)$ such that $U T_i = R_i U$ for all $i=1, \ldots, n$. This is abbreviated as
\[
T \cong R.
\]
The defect space $\cld_{T^*}$ of a row contraction $T$ on $\clh$ is defined by
\[
\cld_{T^*} = \overline{\text{ran}}\Big(I_{\clh} - \sum_{j=1}^{n} T_j T^*_j\Big).
\]
The noncommutative dilation theorem \cite[Theorem 2.1]{Popescue} states the following (the present language is adapted to our current context): There exists an isometry $\Pi_T: \clh \raro \clf^2_n \otimes \cld_{T^*}$ such that
\[
\Pi_T T_i^* = (S_i \otimes I_{\cld_{T^*}})^* \Pi_T,
\]
for all $i=1, \ldots, n$. Moreover, $\clq_T$ is a joint $(S_1^* \otimes I_{\cld_{T^*}}, \ldots, S_n^* \otimes I_{\cld_{T^*}})$-invariant subspace of $\clf^2_n \otimes \cld_{T^*}$, where $\clq_T$ is the \textit{model space} defined by
\[
\clq_T = \Pi_T \clh,
\]
and
\[
T \cong P_{Q_T} S^{D_{T^*}} |_{Q_T}.
\]
Now, in addition, we assume that $(W, T)$ is a contractive representation of $O_n$. In Theorem \ref{thm dil lift} and Corollary \ref{cor model to odometer}, we prove the following dilation (or odometer lift on the model space) theorem:

\begin{theorem}
Let $(W, T)$ be a contractive representation of $O_n$. Then there exists a symbol $L \in \clb(\cld_{T^*}, \clf^2 \otimes \cld_{T^*})$ such that
\[
\Pi_T W^* = W_L^* \Pi_T.
\]
Moreover, $(W,T) \ \cong P_{Q_T} (W_L, S^{D_{T^*}}) |_{Q_T}$.
\end{theorem}

We now turn to invariant subspaces of Fock representations. Invariant subspaces of any natural operator are of interest, and in our case, concrete representations of invariant subspaces would also induce sub-representations of $O_n$. Let $\cle$ be a Hilbert space. By an \textit{invariant subspace of} $\clf^2_n \otimes \cle$, we mean a closed subspace $\cls$ such that
\[
(S_i \otimes I_{\cle}) \cls \subseteq \cls,
\]
for all $i=1, \ldots, n$. In addition, if $L \in \clb(\cle, \clf^2_n \otimes \cle)$ is a symbol, and
\[
W_L \cls \subseteq \cls,
\]
then we say that $\cls$ is an \textit{invariant subspace of the Fock representation} $(W_L, S^\cle)$. Clearly, if $\cls$ is an invariant subspace of $(W_L, S^\cle)$, then $(W_L|_{\cls}, S^\cle|_{\cls})$ on $\cls$ is a contractive representation of $O_n$, where
\[
S^\cle|_{\cls} = ((S_1 \otimes I_{\cle})|_{\cls}, \ldots, (S_n \otimes I_{\cle})|_{\cls}).
\]
We call $(W_L|_{\cls}, S^\cle|_{\cls})$ a \textit{subrepresentation} of the Fock representation $(W_L, S^\cle)$. From this vantage point, it is therefore, natural to ask for a thorough description of subrepresentations of Fock representations. We will provide a clear answer to this question, and to clarify what we mean, we need to discuss the invariant subspaces of vector-valued Fock spaces.

First, we recall the noncommutative Beurling-Lax Halmos theorem \cite[Theorem 2.2]{Popescue char}: Let $\cls \subseteq \clf^2_n \otimes \cle$ be a closed subspace. Then $\cls$ is an invariant subspace if and only if there exist a Hilbert space $\cle_*$ and an inner multi-analytic operator $\Phi \in \clb(\clf^2_n \otimes \cle_*, \clf^2_n \otimes \cle)$ such that
\[
\cls = \Phi(\clf^2_n \otimes \cle_*).
\]
Recall that an operator $\Phi \in \clb(\clf^2_n \otimes \cle_*, \clf^2_n \otimes \cle)$ is inner multi-analytic if $\Phi$ is an isometry and
\[
\Phi (S_i \otimes I_{\cle_*}) = (S_i \otimes I_{\cle}) \Phi,
\]
for all $i=1, \ldots, n$. We use standard techniques to reveal more information about $\Phi$: If $\cls = \Phi(\clf^2_n \otimes \cle_*)$ is an invariant subspace of $\clf^2_n \otimes \cle$ as above, then there exists a unitary $\Pi: \clf^2_n \otimes \cle_* \raro \cls$ such that
\[
\Phi = i_\cls \circ \Pi,
\]
where $i_\cls: \cls \hookrightarrow \clf^2_n \otimes \cle$ is the isometric embedding. Therefore, the following diagram commutes: 

\setlength{\unitlength}{2.8mm}
\begin{center}
\begin{picture}(40,16)(0,0)
\put(15.7,3){$\cls$}\put(20,2.3){$i_\cls$}
\put(23.9,3){$\clf^2_n \otimes \cle$} \put(22, 10){$\clf^2_n \otimes \cle_*$} \put(24,9.2){ \vector(0,-1){5}} \put(22,
9.4){\vector(-1,-1){5.3}} \put(17,3.4){\vector(1,0){5.9}}\put(18,7.2){$\Pi$}\put(24.7,6.5){$\Phi$}
\end{picture}
\end{center}

This additional information about the factorization of inner multi-analytic operator plays a crucial role in proving the following (see Theorem \ref{invariant subspace}):

\begin{theorem}
Let $(W_L, S^\cle)$ be a Fock representation, and let $\cls \subseteq \clf^2_n \otimes \cle$ be a closed subspace. Then $\cls$ is invariant under $(W_L, S^\cle)$ if and only if there exist a Hilbert space $\cle_*$, an inner multi-analytic operator $\Phi \in \clb(\clf^2_n \otimes \cle_*, \clf^2_n \otimes \cle)$, and a Fock representation $(W_{L_*}, S^{\cle_*})$ for some symbol $L_* \in \clb(\cle_*, \clf^2_n \otimes \cle_*)$ such that
\[
\cls = \Phi (\clf^2_n \otimes \cle_*),
\]
and
\[
W_L \Phi = \Phi W_{L_*}.
\]
\end{theorem}

 
In the setting of this theorem, the existence of the inner multi-analytic operator $\Phi$ is due to Popescu. However, the fact that $\Phi$ admits the factorization $\Phi = i_\cls \circ \Pi$ is new. Now, given $\cls$ as an invariant subspace with the inner function $\Phi$ as above, if we also consider the odometer map $W_L$, saying that $\cls$ is invariant under $W_L$ becomes equivalent to finding a unique solution $X \in \clb(\clf^2_n \otimes \cle_*)$ to the operator equation  (see the discussion preceding Theorem \ref{invariant subspace})
\[
W_L \Phi = \Phi X.
\]
Theorem \ref{invariant subspace} establishes that this condition occurs exclusively when $X$ is an odometer map $W_{L_*}$ with a unique symbol $L_* \in \clb(\cle_*, \clf^2_n \otimes \cle_*)$. Furthermore, the operator $L_*$ is defined as
\[
L_* = \Pi^* W_L|_{\cle_*}.
\]

As a special case of the invariant subspace theorem, we have the following answer to the question concerning subrepresentations of Fock representations (see Corollary \ref{cor Unit equiv Fock rep} for more details):

\begin{cor}
Let $\cls$ be an invariant subspace of a Fock representation $(W_L, S^\cle)$. Then the subrepresentation $(W_L|_{\cls}, S^\cle|_{\cls})$ is also a Fock representation.
\end{cor}

We finally remark that the representations of $O_n$ yield pairs of commuting isometries when $n = 1$. Moreover, pairs of doubly commuting isometries arise from Nica-covariant representations of $O_n$.

The remaining sections of the paper are organized as follows: Section \ref{sec Fock repr} is the backbone of this paper, where we completely characterize Fock representations. Section \ref{sec isom Fock rep} provides an extensive classification of isometric Fock representations. Next, in Section \ref{sec Nica}, we present a complete description of Nica-covariant representations. This has been made possible by using explicit representations of conjugates of isometric odometer operators, which are also computed in this section. Section \ref{sec free rep} connects dilation theory and Fock representations. More specifically, we prove that every contractive representation of $O_n$ admits a dilation (or odometer lifting) in the sense of Fock representations. In Section \ref{sec subrepr}, we classify invariant subspaces of Fock representations and essentially prove that subrepresentations of Fock representations are also Fock representations. The final section, Section 7, presents examples of various contractive representations of $O_n$.

\section{Fock representations}\label{sec Fock repr}

Let $\cle$ be a Hilbert space. Recall that
\[
S^{\cle} = (S_1 \otimes I_{\cle}, \ldots, S_n \otimes I_{\cle}),
\]
denotes the $n$-tuple of creation operators on the $\cle$-valued Fock space $\clf_n^2 \otimes \cle$. In this section, we aim at classifying Fock representations of $O_n$ (see Definition \ref{def Fock rep}). In the following, we introduce a new class of operators that depend on symbols that are operators in $\clb(\cle, \clf^2_n \otimes \cle)$. Recall that $\{e_\mu: \mu \in F_n^+\}$ is an orthonormal basis for $\clf^2_n$.

\begin{definition}\label{def odo map}
Let $\cle$ be a Hilbert space and let $L \in \clb(\cle, \clf^2_n \otimes \cle)$. The odometer map $W_L: \clf^2_n \otimes \cle \raro \clf^2_n \otimes \cle$ with symbol $L$ is defined by
\[
W_{L}(\Omega \otimes \eta) = L\eta,
\]
for all $\eta \in \cle$, and for each $e_\mu = e_{{\mu_{1}}} \otimes\cdots \otimes e_{\mu_{m}}$, $m \geq 1$, define
\[
W_{L}(e_{\mu} \otimes \eta) =
\begin{cases}
e_{\mu_1+1} \otimes e_{\mu_2}\otimes\cdots \otimes e_{\mu_{m}} \otimes \eta & \mbox{if } \mu_{1} \neq n
\\
e_{{1}} \otimes e_{{\mu_{2}+1}} \otimes e_{\mu_3} \otimes \cdots \otimes e_{\mu_{m}} \otimes \eta & \mbox{if } \mu_{1}=n, \mu_{2} \neq n
\\
e^{\otimes 2}_{1}   \otimes e_{{\mu_{3}+1}} \otimes e_{\mu_4} \otimes \cdots \otimes e_{\mu_{m}} \otimes \eta & \mbox{if } \mu_{1}=\mu_{2}=n, \mu_{3} \neq n
\\
\vdots & \vdots
\\
e^{\otimes m}_{1} \otimes L \eta & \mbox{if } \mu_{1}=\cdots= \mu_{m}=n.
\end{cases}
\]
\end{definition}

A couple of comments are in order. Since $\{e_\mu: \mu \in F_n^+\}$ is an orthonormal basis for $\clf^2_n$, it follows that $W_L \in \clb(\clf^{2}_n \otimes \cle)$ (see Remark \ref{remark odometer maps} for more details). Second, the unitary odometer maps were introduced by Li in \cite{Boyu Li22}. The fundamental difference in his case is that the symbol $L$ is a self-map, that is, $L \in \clb(\cle)$ (instead of $L \in \clb(\cle, \clf^2_n \otimes \cle)$). Moreover, in the case of unitary odometer maps, Li proved that the symbols become unitary. This will be observed again in Theorem \ref{W_L uni}. Before we go any further, we introduce the notion of the length of words. The \textit{length} of the word $\mu \in F_n^+$ is defined by	
\[
\vert \mu\vert =
\begin{cases}
k & \mbox{if } \mu = g_{\mu_{1}}g_{\mu_{2}}\hdots g_{\mu_{k}} \\
0 & \mbox{if } \mu = g_{0}.
\end{cases}
\]
Now we are ready for the general classification and representations of odometer maps:

\begin{theorem}\label{Prop: Odometer operator}
Let $\cle$ be a Hilbert space, and let $W \in \clb(\clf^2_n \otimes \cle)$. Then $(W, S^{\cle})$ is a Fock representation if and only if there is a unique symbol $L \in \clb(\cle,  \clf^2_n \otimes \cle)$ such that
\[
W = W_L.
\]
\end{theorem}
\begin{proof} Suppose $(W, S^{\cle})$ is a Fock representation. Define $L:\cle \to \clf^2_n \otimes \cle$ by
\[
L \eta = W(\Omega \otimes \eta),
\]
for all $\eta \in \cle$. Uniqueness of $L$ is clear by this definition. We claim that $W=W_L$. To this end, fix $\eta \in \cle$ and $\mu \in {F}^{+}_{n}$. Then for any $k < n$, we have $W (S_k \otimes I_\cle) = S_{k+1} \otimes I_\cle$ and hence
\[
W(e_{k}\otimes e_{\mu} \otimes \eta) = W  (S_k \otimes I_\cle) (e_{\mu} \otimes \eta) = (S_{k+1} \otimes I_\cle) (e_{\mu} \otimes \eta),
\]
which yields
\[
W(e_{k}\otimes e_{\mu} \otimes \eta) = e_{k+1}\otimes e_{\mu} \otimes \eta.
\]
On the other hand, by the definition of $W_L$, we know that $W_L(e_{k}\otimes e_{\mu} \otimes \eta) = e_{k+1}\otimes e_{\mu} \otimes \eta$, and consequently
\begin{equation}\label{eqn W ek}
W(e_{k}\otimes e_{\mu} \otimes \eta)= W_L(e_{k}\otimes e_{\mu} \otimes \eta).
\end{equation}
Now, let $m \geq 1$, and assume that $\mu_1 \neq n$. Since $W (S_n^m \otimes I_\cle) = (S_1^m \otimes I_\cle) W$, it follows that
\[
W(e_{n}^{\otimes m} \otimes e_{\mu} \otimes \eta) = W (S_n^m \otimes I_\cle)(e_{\mu} \otimes \eta) = e_{1}^{\otimes m} \otimes W(e_{\mu}\otimes \eta).
\]
We have two cases to consider: Let $|\mu|=0$, that is, $\mu = g_0$. Then $W(e_{\mu}\otimes \eta) = W(\Omega \otimes \eta) = L\eta$. By the definition of $W_L$, we have that
\[
W(e_{n}^{\otimes m} \otimes \eta) = W_L(e_{n}^{\otimes m} \otimes \eta).
\]
For the next case, assume that $|\mu| > 0$, and let $e_\mu = e_{\mu_1} \otimes \cdots \otimes e_{\mu_k}$, where $k \geq 1$. By \eqref{eqn W ek}, we know that
\[
W(e_{\mu}\otimes \eta) = W_L(e_{\mu}\otimes \eta).
\]
Combining these two cases, we finally conclude that
\[
W(e_{n}^{\otimes m} \otimes e_{\mu} \otimes \eta) = W_L(e_{n}^{\otimes m} \otimes e_{\mu} \otimes \eta),
\]
for all $\mu \in F_n^+$, which completes the proof of the fact that $W = W_L$.

\noindent For the converse, assuming that $W=W_{L}$ for some symbol $L \in \clb(\cle, \clf^{2}_n \otimes \cle)$, we want to prove that $(W, S^{\cle})$ is a Fock representation, that is, to prove that $W_L (S_k \otimes I_\cle) = S_{k+1} \otimes I_\cle$ for all $k<n$, and $W_L (S_n \otimes I_\cle) = (S_1 \otimes I_\cle) W_L$. Fix $\eta \in \cle$ and $\mu \in {F}^{+}_{n}$. For any $k<n$, since $W_L (e_k \otimes \eta) = e_{k+1} \otimes \eta$, we have
\[
W_L (S_{k} \otimes I_\cle) (e_{\mu} \otimes \eta) = W_L(e_{k}\otimes e_{\mu} \otimes \eta) = e_{k+1}\otimes e_{\mu} \otimes \eta = (S_{k+1} \otimes I_\cle) (e_{\mu} \otimes \eta),
\]
which means
\[
W_L (S_{k} \otimes I_\cle) = S_{k+1} \otimes I_\cle,
\]
for all $k< n$. Now we prove that
\[
W_L (S_{n} \otimes I_\cle) = (S_{1} \otimes I_\cle)W_L.
\]
Observe that
\[
(S_1 \otimes I_\cle) W_L (e_{\mu} \otimes \eta) = e_{1}\otimes W_L (e_{\mu}\otimes \eta) = W_L (e_{n}\otimes e_{\mu} \otimes \eta) = W_L (S_n \otimes I_\cle)(e_{\mu}\otimes \eta).
\]
Summarizing, we have $W_L (S_{n} \otimes I_\cle) = (S_{1} \otimes I_\cle)W_L$. This along with $W_L (S_{k} \otimes I_\cle) = S_{k+1} \otimes I_\cle$ for all $k<n$, implies that $(W_L, S^{\cle})$ is a Fock representation.
\end{proof}

In particular, in the scalar case, we have the following:

\begin{cor}\label{Prop: Odometer operator to scalar case}
Let $W \in \clb(\clf^2_n)$. Then $(W, S)$ is a Fock representation if and only if there is unique $\xi \in  \clf^2_n$ such that $W=W_{\xi}.$
\end{cor}
\begin{proof}
The proof is immediate, given that a bounded linear operator $L: \mathbb{C} \raro \clf^2_n$ can be represented as $L1 = \xi$ for some $\xi \in \clf^2_n$.
\end{proof}

A symbol $L \in  \clb(\cle, \clf^2_n \otimes \cle)$ is said to be \textit{constant} if
\[
L\cle \subseteq \Omega \otimes \cle.
\]
From the above theorem, it is clear that $L$ is a constant symbol if and only if the odometer map $W_L$ satisfies the following condition:
\[
W_L(\Omega \otimes \cle) \subseteq \Omega \otimes \cle.
\]
In the following sections, we will focus on more specific representations within the category of Fock representations. For instance, we will discuss the matter of constant unitary symbols and revisit the flavor of Li's original finding of unitary odometer maps.

\begin{remark}\label{remark odometer maps}
It should be noted that odometer maps are clearly defined for any operator (which we call a symbol) belonging to $\clb(\cle, \clf^2_n \otimes \cle)$. This bears a resemblance to Popescu's concept of noncommutative analytic Toepliz operators \cite{Popescu 95}. However, it is crucial to note that noncommutative analytic Toepliz operators are not defined for all symbols in $\clb(\cle, \clf^2_n \otimes \cle)$. The odometer maps display a clear contrast when viewed from this perspective. We also have additional features about these maps:
\[
\|L\| \leq \| W_{L} \| \leq 1 + \| L \|.
\]
Indeed, for $\eta \in \cle$, we have

\[
\|L\eta\| = \|W_L(\Omega \otimes \eta)\| \leq \|W_L\| \|\eta\|,
\]
which implies that $\|L\| \leq \|W_L\|$. On the other hand, define \( L': \mathbb{C} \to \mathcal{F}^2_n \) by
\[
L' (1) = \Omega.
\]
Denote the corresponding operator \( W_{L'} \in \mathcal{B}(\mathcal{F}^2_n) \) by \( W_{\Omega} \). Set
\[
M = \overline{\text{span}} \{ e^{m}_{n} \otimes \eta : m \geq 0, \eta \in \mathcal{E} \}.
\]
Then
\[
W_{L} = (W_{\Omega} \otimes I) P_{M^{\perp}} + (W_{\Omega} \otimes L) P_{M},
\]
implying that $\| W_{L} \| \leq 1 + \| L \|$.
\end{remark}

We conclude with an example demonstrating that the norm of $L$ can be strictly less than the norm of $W_L$: Let $\cle$ be a Hilbert space. Define $L: \cle \to \mathcal{F}^2_2\otimes \cle$ by
$$
L \eta = \dfrac{1}{\sqrt{2}} (e_1 +e^{\otimes2}_1)\otimes \eta,
$$
for all $\eta\in \cle$. Clearly, $L$ is an isometry and hence $\norm{L}=1$. Fix a unit vector $\eta \in \cle$. Then $\norm{x}=1$, where
\[
x=\dfrac{1}{\sqrt{2}} (e_2+e^{\otimes2}_2)\otimes \eta.
\]
Observe that
\[
W_L x= \dfrac{1}{\sqrt{2}}(e_1 +e_1^{\otimes 2} )\otimes L\eta = \dfrac{1}{\sqrt{2}}(e_1 +e_1^{\otimes 2} )\otimes \dfrac{1}{\sqrt{2}} (e_1 +e^{\otimes2}_1)\otimes \eta =\dfrac{1}{2}(e_1^{\otimes 2}+ 2 e_1^{\otimes 3}+ e_1^{\otimes 4})\otimes \eta,
\]
which implies that
\[
\norm{W_L(x)}^2=\dfrac{3}{2},
\]
and hence
\[
\norm{W_L}\geq \sqrt{\dfrac{3}{2}}>1=\norm{L}.
\]

\section{Isometric Fock representations}\label{sec isom Fock rep}

In this section, we offer a thorough classification of isometric Fock representations. Recall that an isometric Fock representation is a Fock representation $(W, S^{\cle})$ where $W \in \clb(\clf^2_n \otimes \cle)$ is an isometry.

We note a simple fact that will be significant in the subsequent background computation. We partition the orthonormal basis $\{e_\mu: \mu \in F_n^+\}$ into two disjoint sets as follows: 
\begin{equation}\label{eqn partition}
\{e_\mu: \mu \in F_n^+\} = \{e_1^{\otimes m}: m \in \Z_+\} \bigsqcup \{ e_{\mu}: \mu \in F^+_n, |\mu| \geq 1, \mu_i \neq 1 \text{ for some } i\}.
\end{equation}
Throughout this section, we fix a Hilbert space $\cle$ and an orthonormal basis $\{h_p\}_{p\in \Lambda}$ for $\cle$. Note that $\Lambda$ is either a finite or a countably infinite set. Clearly, $\{e_{\mu} \otimes {h_p}: \mu \in F_n^+, p \in \Lambda\}$ is an orthonormal basis for $\clf^2_n \otimes \cle$, and hence
\[
\clf^2_n \otimes \cle = \overline{\rm span}\{ {e_{\mu} \otimes {h_p}} : \mu \in {F}^{+}_{n}, p\in \Lambda\}.
\]
Moreover, given $L\in \mathcal{B} (\cle, \clf^2_n \otimes \cle)$, we define the closed subspace
\begin{equation}\label{eqn E_L}
\cle_L = \overline{\rm span}\{ {e^{\otimes{m}}_{1} \otimes {\eta}} : m \in \mathbb{Z}_{+}, \eta \in \cle\} \ominus \overline{\rm span}\{ {e^{\otimes{p}}_{1} \otimes {L \zeta}} : p \geq 1, \zeta \in \cle\}.
\end{equation}
We are now ready for the classification of isometric Fock representations.

\begin{theorem}\label{ W_L iso iff L}
Let $L\in \mathcal{B} (\cle, \clf^2_n \otimes \cle)$. Then $(W_L, S^{\cle})$ is an isometric Fock representation if and only if the following conditions hold:
\begin{enumerate}
\item $L$ is an isometry.
\item $L \cle \subseteq \cle_L$.
\end{enumerate}
\end{theorem}
\begin{proof}
Suppose $W_L$ is an isometry, and let $\eta \in \cle$. Since
\[
\|L\eta\| = \|W_L (\Omega \otimes \eta)\| = \|\Omega \otimes \eta \| = \|\eta\|,
\]
it follows that $L$ is an isometry. For (2), we apply the partition in \eqref{eqn E_L}. Define a closed subspace $\clm$ of $\clf^2_n \otimes \cle$ as
\[
\clm =\overline{\rm span}\{ {e_{\mu} \otimes {\eta}} : \eta \in \cle, \mu \in F^+_n, |\mu| \geq 1, \mu_i \neq 1 \text{ for some } i\}.
\]
Clearly, $\clm^{\perp} = \overline{\rm span}\{ e^{\otimes{m}}_{1} \otimes {\eta}: m \in \Z_+, \eta \in \cle\}$. Pick $\eta, \gamma \in \cle$, and $e_\mu = e_1^{\otimes m} \otimes e_{\mu_{1}} \otimes  \cdots \otimes e_{\mu_{k}}$. Assume that $\mu_1 \neq 1$ and $m \in \Z_+$, so that $e_\mu \otimes \gamma \in \clm$. Since $W_L$ is an isometry, it follows that
\[
\langle L \eta,  e_{\mu} \otimes {\gamma} \rangle =\langle W_L (\Omega \otimes \eta),  W_L(e_n^{\otimes m} \otimes e_{{\mu_{1}}-1} \otimes  \cdots \otimes e_{\mu_{k}} \otimes \gamma) \rangle = 0,
\]
and hence $L\cle  \subseteq \clm^{\perp} = \overline{\rm span}\{ e^{\otimes{m}}_{1} \otimes {\eta}: m \geq 0, \eta \in \cle\}$. Next, for any $p \geq 1$, the fact that $W_L$ is an isometry again implies that
\[
\langle L \eta,  e_{1}^{\otimes p} \otimes L {\gamma} \rangle = \langle W_L (\Omega \otimes \eta),  W_L(e_n^{\otimes p} \otimes \gamma) \rangle =0,
\]
and hence
\[
L\cle \perp \overline{\rm span}\{ {e^{\otimes{p}}_{1} \otimes {L \zeta}} : p \geq 1, \zeta \in \cle\}.
\]
This, along with $L\cle  \subseteq \clm^{\perp}$, implies that $L\cle \subseteq \cle_L$.

\noindent For the converse direction, again, in view of the partition in \eqref{eqn partition}, we write $\clf^2_n = \overline{\text{span}}(\clm_1 \cup \clm_2)$, where
\[
\clm_1= \{ e^{\otimes{m}}_{n}: m \geq 0\},
\]
and
\[
\clm_2 = \{ {e_{\mu}}: \mu \in F^+_n, |\mu| \geq 1, \mu_i \neq n \text{ for some } i\}.
\]
Pick $e_{\mu} \otimes \eta \in \clf^2_n \otimes \cle$. If $e_{\mu} =e_{n}^{\otimes p}\in \clm_1$, then
\[
\|W_L(e_{n}^{\otimes p}\otimes \eta)\| = \|e_{1}^{\otimes p} \otimes L \eta\| = \|L \eta \| = \|\eta \| = \| e_{n}^{\otimes p}\otimes \eta\|.
\]
If $e_{\mu} =e_n^{\otimes m} \otimes e_{\mu_{1}} \otimes  \cdots \otimes e_{\mu_{k}} \in \clm_2$, $\mu_{1} \neq n$, then
\[
\norm{W_L(e_{\mu} \otimes \eta)} = \|e_1^{\otimes m} \otimes e_{\mu_{1}+1} \otimes e_{\mu_{2}} \otimes  \cdots \otimes e_{\mu_{k}} \otimes \eta\| = \norm{\eta} = \norm{e_{\mu}\otimes \eta}.
\]
As a result, we proved that $W_L$ is isometric while acting on every basis vector. To conclude that $W_L$ is isometry, it clearly remains to prove that
\begin{equation}\label{eqn inner = 0}
\langle {W_L(e_{\mu} \otimes \eta}), {W_L(e_{\lambda} \otimes \zeta}) \rangle =0,
\end{equation}
whenever $\langle e_{\mu} \otimes \eta , e_{\lambda} \otimes \zeta \rangle=0$ for $\mu, \lambda \in F_n^+$ and $\eta, \zeta \in \cle$. The latter condition is equivalent to $\langle e_{\mu}, e_{\lambda} \rangle=0$ or $\langle \eta , \zeta \rangle =0$. Suppose  $\langle e_{\mu}, e_{\lambda} \rangle=0$. Let $|\mu|=p, |\lambda|=q$. We have the following three cases:

\noindent \textit{Case I:} Let $e_{\mu}, e_{\lambda} \in \clm_1$. Then $e_{\mu}= e^{\otimes p}_{n}$ and  $e_{\lambda}=e^{\otimes q}_{n}$. Since $\langle e_{\mu}, e_{\lambda} \rangle=0$, we conclude that $p \neq q$. Assume, without any loss of generality, that $p < q$. Then
\[
\begin{split}
\langle W_L(e_{\mu} \otimes \eta), W_L(e_{\lambda} \otimes \zeta)  \rangle & = \langle W_L(e^{\otimes p}_{n} \otimes \eta), W_L(e^{\otimes q}_{n}\otimes \zeta) \rangle
\\
& = \langle e^{\otimes p}_1 \otimes L \eta, e^{\otimes q}_1 \otimes L \zeta \rangle
\\
& = \langle L\eta, e^{\otimes (q-p)}_{1}  \otimes L \zeta \rangle
\\
& = 0,
\end{split}
\]
as $L\cle \subseteq \cle_L$. This proves \eqref{eqn inner = 0}.

\noindent \textit{Case II:} Let $e_{\mu} \in \clm_i$ and $e_{\lambda} \in \clm_j$, where $i\ne j$. Without loss of generality, we assume that $e_{\mu} \in \clm_1$ and $e_{\lambda} \in \clm_2$  (as other cases will follow similarly). Then $e_{\mu} = e^{\otimes p}_{n}$ and $e_{\lambda}= e_n^{\otimes m} \otimes e_{\mu_{1}} \otimes  \cdots \otimes e_{\mu_{k}}$ and $\mu_i \neq n$ for some $i$. Then $W_L(e_{\lambda} \otimes \zeta)$ has one component different from $e_1$, whereas $L \cle \subseteq \cle_L$ implies that
\[
W_L(e_{\mu} \otimes \eta) \in \overline{\rm span}\{{e^{\otimes{m}}_1 \otimes {\eta}}: m \in \mathbb{Z}_+, \eta \in \cle\}.
\]
Therefore, \eqref{eqn inner = 0} holds.

\noindent \textit{Case III:}  Let $e_{\mu}, e_{\lambda} \in \clm_2$. Here, we have two subcases to deal with: $p=q$ and $p <q$.\\
If $p=q,$ then there exists a minimum $i$ such that $\mu_i \neq \lambda_i$. Therefore, \eqref{eqn inner = 0} holds because the definition of $W_L$ implies that they will differ at the $i^{th}$ component. If $p <q$, then again, the definition of $W_L$ implies that $W_L (e_{\mu} \otimes \eta) \in ({\mathbb{C}}^n)^{\otimes p} \otimes \cle$ and $W_L (e_{\lambda} \otimes \zeta) \in ({\mathbb{C}}^n)^{\otimes q} \otimes \cle$, proving \eqref{eqn inner = 0}.

\noindent Finally, assume that $\langle e_{\mu}, e_{\lambda} \rangle \neq 0$. Then $\langle \eta, \zeta \rangle =0$ and $\mu = \lambda$. We have the following two cases:

\noindent \textit{Case A:} $e_{\mu} \in \clm_1$: Let $e_{\mu}= e_n^{\otimes p}$. Then
\begin{align*}
\langle W_L(e_{\mu} \otimes \eta), W_L( e_{\mu} \otimes \zeta) \rangle = \langle e_1^{\otimes p} \otimes L \eta, e_1^{\otimes p}  \otimes L \zeta \rangle = \|e_1^{\otimes p}\|^2  \langle L \eta,  L \zeta \rangle = 0,
\end{align*}
where the last equality follows from the fact that $L$ is an isometry.

\noindent \textit{Case B:} $e_{\mu} \in \clm_2$: Let $e_{\mu}= e_n^{\otimes m} \otimes e_{\mu_{1}} \otimes \cdots \otimes e_{\mu_k}$, where $\mu_1 \neq n$. Then
\begin{align*}
\langle W_L(e_{\mu} \otimes \eta), W_L( e_{\mu} \otimes \zeta) \rangle
&= \langle e_1^{\otimes m} \otimes e_{\mu_{1}+1} \otimes \cdots \otimes e_{\mu_k} \otimes \eta, e_1^{\otimes m} \otimes e_{\mu_{1}+1} \otimes \cdots \otimes e_{\mu_k}\otimes \zeta \rangle \\
&= \norm{ e_1^{\otimes m} \otimes e_{\mu_{1}+1} \otimes \cdots \otimes e_{\mu_k}}^2 \langle \eta,  \zeta \rangle \\ &= 0.
\end{align*}
The above two cases also assert that $\langle \eta, \zeta  \rangle  = 0$ implies \eqref{eqn inner = 0}, which consequently completes the proof of the theorem.
\end{proof}

Let $(W_L, S^{\cle})$ be a Fock representation. By condition (2) of the above theorem, we know that $L \cle \subseteq \cle_L$. We claim that this implies the following identity:
\begin{equation}\label{eqn: equiv to 3}
\sum_{p \in \Z_+, q \in \Lambda}c^{\eta}_{{p+r},q}\overline{c^{\zeta}_{p,q}}= 0,
\end{equation}
for all $r \geq 1$, where
\[
L \eta = \sum_{s \in \Z_+, t \in \Lambda} c^{\eta}_{s,t} e^{\otimes s}_1\otimes h_t, \mbox{ and } L \zeta = \sum_{p \in \Z_+, q \in \Lambda} c^{\zeta}_{p,q} e^{\otimes p}_1\otimes h_q,
\]
and $\{h_p\}_{p \in \Lambda}$ is an orthonormal basis for $\cle$. Indeed, for every $r\geq1$, we compute
\begin{align*}
\Big\langle {L \eta},  e^{\otimes r}_{1}\otimes L \zeta \Big\rangle & =  \Big\langle \sum_{s \in \Z_+, t \in \Lambda} c^{\eta}_{s,t} e^{\otimes s}_1\otimes h_t, e^{\otimes r}_{1}\otimes \sum_{p \in \Z_+, q \in \Lambda} c^{\zeta}_{p,q} e^{\otimes p}_1\otimes h_q\Big\rangle
\\
&=  \Big\langle  \sum_{s \in \Z_+, t \in \Lambda} c^{\eta}_{s,t} e^{\otimes s}_1\otimes h_t,  \sum_{p \in \Z_+, q \in \Lambda} c^{\zeta}_{p,q} e^{\otimes (p+r)}_1\otimes h_q\Big\rangle
\\
&= \sum_{p \in \Z_+, q \in \Lambda}c^{\eta}_{{p+r},q}\overline{c^{\zeta}_{p,q}} \norm{h_q}^2
\\
&=  \sum_{p \in \Z_+, q \in \Lambda} c^{\eta}_{{p+r},q}\overline{c^{\zeta}_{p,q}},
\end{align*}
as $\{h_q\}_{q \in \Lambda}$ forms an orthonormal basis for $\cle$. Therefore, $\langle {L \eta},  e^{\otimes r}_{1}\otimes L \zeta \rangle = 0$ for all $r\geq 1$ and all $\eta, \zeta \in \cle$ implies \eqref{eqn: equiv to 3}.

The characterization of isometric Fock representations obtained in Theorem \ref{ W_L iso iff L} is more explicit in the scalar case. Recall that the odometer maps on $\clf^2_n$ are of the form $W_\xi$, $\xi \in \clf^2_n$. In this case, we will represent $c^{\eta}_{{p},q}$ in \eqref{eqn: equiv to 3} simply as $c_p$.

\begin{cor}\label{cor scalar iso Fock rep}
Let $\xi \in \clf^2_n$.  Then  $W_{\xi}$ is an isometry if and only if
\[
\xi = \sum^{\infty}_{p=0} c_{p}e^{\otimes p}_{1},
\]
where $\sum^{\infty}_{p=0} \vert c_{p}\vert^2=1$ and $\sum^{\infty}_{p=0}  c_{p+r} \overline{c_{p}}=0$ for all $r\geq 1$.
\end{cor}
\begin{proof}
Let $W_{\xi}$ be an isometry. Theorem \ref{ W_L iso iff L} implies that $\xi \in \overline{\rm span}\{ {e^{\otimes{m}}_{1} } : m\in \mathbb{Z}_{+}\}$. There exists a sequence of scalars $\{c_p\}_{p \in \Z_+}$ such that
\[
\xi = W_{\xi}(\Omega)= \sum_{p \in \Z_+} c_{p} e^{\otimes p}_1.
\]
Also, $W_{\xi}$ is an isometry, which implies
\[
1= \norm {\Omega}^2=\norm{W_{\xi}(\Omega)}^2 = \norm{\xi}^2=\sum_{p \in \Z_+}|c_{p}|^2.
\]
The remaining condition directly follows from the scalar case of \eqref{eqn: equiv to 3}. We can derive the converse direction from the vector setting of Theorem \ref{ W_L iso iff L} by considering the map $L:\mathbb{C} \to \clf^2_n$ defined by $L 1 = \xi$.
\end{proof}

A number of questions arise about odometer maps. Within this section, our accomplishments were limited to the classification of isometric Fock representations. While this may be adequate for our needs, it can be challenging to find out the overall structure of odometer maps. Calculating the conjugate of an odometer map may appear to be a simple question, yet the solution is unclear. In the following section, however, we will compute the conjugate of isometric odometer maps.

\section{Nica covariant representations}\label{sec Nica}

This section aims at classifying Nica-covariant representations of $O_n$. Recall from Definition \ref{def Nica} that a Nica-covariant representation of $O_n$ is an isometric Fock representation $(W, S^\cle)$ such that
\[
W^*(S_1 \otimes I_\cle) = (S_n \otimes I_\cle) W^*.
\]
Let $\cle$ be a Hilbert space. Suppose $(W,S^\cle)$ on $\clf^2_n \otimes \cle$ is a Nica covariant representation of $O_n$. By Theorem \ref{Prop: Odometer operator}, $W$ is an odometer map, that is, there is a unique symbol $L \in \clb(\cle, \clf^2_n \otimes \cle)$ such that $W = W_L$. Therefore, the revised goal of this section is to classify the symbol $L$ such that the representation $(W_L, S^\cle)$ is a Nica covariant representation. At this stage, we need representations of the adjoints of isometric odometer maps.

In what follows, given a Hilbert space $\cle$, we always assume that $\{h_p\}_{p \in \Lambda}$ is an orthonormal basis for $\cle$. Moreover, for all $q \in \Lambda$, we write the Fourier series representation
\begin{equation}\label{eqn L hq}
Lh_q = \sum_{r \in \Z_+, s \in \Lambda} c^{h_q}_{r,s} e^{\otimes r}_1\otimes h_s.
\end{equation}
It is also worth recalling the partition of the orthogonal basis $\{e_\mu: \mu \in F_n^+\}$ for $\clf^2_n$, as in \eqref{eqn partition}.

\begin{prop}\label{Lemma: Adjoint of W}
Let $\cle$ be a Hilbert space, and let $L \in \clb(\cle, \clf^2_n \otimes \cle)$. If the odometer map $W_L$ is an isometry, then the adjoint of $W_L$ is given by
\[
W_L^* f =
\begin{cases}
\displaystyle\sum^{m}_{p=0} \displaystyle\sum_{q \in \Lambda} \overline{c^{h_q}_{m-p,l}} (e_{n}^{\otimes p}\otimes h_q) & \mbox{if } f =\  e_{1}^{\otimes m}\otimes h_l
\\
e_{n}^{\otimes m} \otimes e_{\mu_{1}-1}\otimes e_{\mu_{2}} \cdots \otimes e_{\mu_k} \otimes h_l & \mbox{if } f = e_{1}^{\otimes m} \otimes e_{\mu_{1}}\otimes \cdots \otimes e_{\mu_{k}} \otimes h_l \; \text{and }\mu_1 > 1,
\end{cases}
\]
for all $m \in \Z_+$ and $l \in \Lambda$.
\end{prop}
\begin{proof}
Fix $m \in \Z_+$, $l \in \Lambda$, and let $\mu = g_{\mu_1} \cdots g_{\mu_k}$, $k\geq 1$. For all $q \in \Lambda$, we have
\[
\langle  W_L^*(e_{1}^{\otimes m} \otimes h_l), e_{\mu}  \otimes h_q \rangle = \langle e_{1}^{\otimes m} \otimes h_l,  W_L(e_{\mu}  \otimes h_q) \rangle =0,
\]
whenever $\mu_{i} \neq n$ for some $i$ or $|\mu| \geq m+1$. Next, assume that $|\mu|\leq m$. For any $0 \leq p \leq m $, we have
\[
\langle W_L^*(e_{1}^{\otimes m}\otimes h_l), e_{n}^{\otimes p}\otimes h_q  \rangle = \langle e_{1}^{\otimes m}\otimes h_l, W_L (e_{n}^{\otimes p}\otimes h_q ) \rangle = \langle e_{1}^{\otimes m}\otimes h_l,  e_{1}^{\otimes p}\otimes L h_q \rangle.
\]
Substituting the value of $L h_q$ yields (see \eqref{eqn L hq})
\begin{align*}
\langle e_{1}^{\otimes m}\otimes h_l,  e_{1}^{\otimes p}\otimes L h_q \rangle
& = \langle e_{1}^{\otimes m}\otimes h_l,  e_{1}^{\otimes p}\otimes \sum_{r \in \Z_+, s \in \Lambda} c^{h_q}_{r,s} e^{\otimes r}_1\otimes h_s \rangle
\\
& =\langle e_{1}^{\otimes m}\otimes h_l, \sum_{r \in \Z_+, s \in \Lambda} c^{h_q}_{r,s} e^{\otimes (r+p)}_1\otimes h_s \rangle\\
& = \langle e_{1}^{\otimes m}\otimes h_l, \sum^{\infty}_{r=p}\sum_{s \in \Lambda} c^{h_q}_{r-p,s} e^{\otimes r}_1\otimes h_s \rangle
\\
&=\overline{c^{h_q}_{m-p,l}},
\end{align*}
which implies
\[
\begin{split}
W_L^*(e_{1}^{\otimes m}\otimes h_l) = \sum^{m}_{p =0}\sum_{q \in \Lambda}\langle W_L^*(e_{1}^{\otimes m}\otimes h_l), e_{n}^{\otimes p}\otimes h_q  \rangle (e_{n}^{\otimes p}\otimes h_q) = \sum^{m}_{p=0}\sum_{q \in \Lambda} \overline{c^{h_q}_{m-p,l}} (e_{n}^{\otimes p}\otimes h_q).
\end{split}
\]
Finally, assume that $f=e_{1}^{\otimes m} \otimes e_{\mu_{1}} \otimes \cdots \otimes e_{\mu_{k}} \otimes \eta$, where $\mu_1 > 1$ and $\eta \in \cle$. As $W_L^*W_L = I$, we have
\begin{align*}
W_L^* f & = W_L^*(e_{1}^{\otimes m}\otimes e_{\mu_{1}} \otimes \cdots \otimes e_{\mu_{k}}  \otimes \eta)
\\
&= W_L^*W_L (e_{n}^{\otimes m}\otimes e_{\mu_{1}-1} \otimes e_{\mu_{2}} \otimes \cdots \otimes e_{\mu_{k}} \otimes \eta)
\\
&= e_{n}^{\otimes m}\otimes e_{\mu_{1}-1} \otimes e_{\mu_{2}} \otimes \cdots \otimes e_{\mu_{k} }\otimes \eta,
\end{align*}
which completes the proof of the proposition.
\end{proof}

Recall from Corollary \ref{cor scalar iso Fock rep} that if $W_{\xi}: \clf^2_n\to \clf^2_n$ is an isometry for some $\xi \in \clf^2_n$, then, in particular, we have
\[
\xi = \sum^{\infty}_{p=0} c_{p}e^{\otimes p}_{1}.
\]
In this case, the above proposition yields the adjoint formula of $W_\xi$ as (note that here $\xi = L1$, and hence \eqref{eqn L hq} is comparable)
\[
W_{\xi}^* f =
\begin{cases}
\overline{c_{0}} \Omega  & \mbox{if } f =\Omega
\\
\displaystyle\sum_{p=0}^{m} \overline{c_{m-p}}e^{\otimes p}_{n} & \mbox{if } f =e^{\otimes m}_1, m\geq 1
\\
e_{n}^{\otimes m} \otimes e_{\mu_{1}-1}\otimes e_{\mu_{2}} \cdots \otimes e_{\mu_k} & \mbox{if } f = e_{1}^{\otimes m} \otimes e_{\mu_{1}}\otimes \cdots \otimes e_{\mu_{k}}, \mu_1 > 1, \text{and } m\geq 0.
\end{cases}
\]

Now we are ready to characterize Nica-covariant representations of $O_n$:
		
\begin{theorem}\label{Thm: Nica covariant characterization}
Let $\cle$ be a Hilbert space, and let $L \in \clb(\cle, \clf^2_n \otimes \cle)$. Assume that $(W_L, S^\cle)$ is an isometric Fock representation. Then $(W_L, S^\cle)$ is a Nica covariant representation of $O_n$ if and only if
\[
L \cle \subseteq \Omega \otimes \cle.
\]
\end{theorem}
\begin{proof}
Let $(W_L, S^\cle)$ be a Nica-covariant representation of $O_n$. By \cite[Theorem 4.4]{Boyu Li22}, we have
\[
\clf^2_n \otimes \cle=\clh_{us} \oplus \clh_{ss},
\]
where
\[
{\clh}_{us} = \bigoplus_{\mu \in F^+_n} \left( S_{\mu} \otimes I_{\cle}\right) \left( \bigcap_{i=1}^{n} \ker (S_i \otimes I_{\cle})^* \cap {\clh}^{W_L}_{u} \right),
\]
and
\[
{\clh}_{ss} = \bigoplus_{\mu \in F^+_n} \left( S_{\mu} \otimes I_{\cle}\right) \left( \bigcap_{i=1}^{n} \ker (S_i \otimes I_{\cle})^* \cap {\clh}^{W_L}_{s} \right).
\]
Here, ${\clh}^{W_L}_{u}$ and ${\clh}^{W_L}_{s}$ denote the unitary and the shift part of $W_L$, respectively. We know that
\[
\bigcap_{i=1}^{n} \ker (S_i \otimes I_{\cle})^* =\Omega \otimes \cle.
\]
Let
\[
(\Omega \otimes \cle) \cap {\clh}^{W_L}_{u} = \Omega \otimes \cle_1,
\]
and
\[
(\Omega \otimes \cle) \cap {\clh}^{W_L}_{s} = \Omega \otimes \cle_2,
\]
Note that $\cle_1$ and $\cle_2 $ are subspaces of $\cle$ and $\cle_1 \oplus \cle_2= \cle$. Also, the Nica-covariant conditions imply that
\[
L(\cle_i) = W_L(\Omega \otimes \cle_i) \subseteq \Omega \otimes\cle_i,
\]
for $i=1,2$. Given $\eta \in \cle$, we write $\eta = \eta_1 + \eta _2$ for some $\eta_1 \in \cle_1$ and $\eta_2 \in \cle_2$. Then
\[
L \eta = L \eta_1 + L \eta_2 = W_L(\Omega \otimes \eta_1) +W_L(\Omega \otimes \eta_2) \in \Omega \otimes \cle_1 \oplus \Omega \otimes \cle_2 = \Omega \otimes \cle,
\]
implies that $L(\cle) \subseteq \Omega \otimes \cle$.

\noindent Conversely, let $L(\cle) \subseteq \Omega \otimes \cle$, that is, $L: \cle \rightarrow  \Omega \otimes \cle$ is an isometry. We consider $L$ as an isometry from $\cle$ to itself. Then the Wold decomposition of $L$ gives us the orthogonal decomposition of closed subspaces
\[
\cle=\cle_1 \oplus \cle_2,
\]
where $L_1 :=L|_{\cle_1}$ is unitary and $L_2:= L|_{\cle_2}$ is a shift. Then $L=L_1 \oplus L_2,$ and $W_L= W_{L_1} \oplus W_{L_2}$. As $L_1$ is unitary, therefore $W_{L_1}$ is also a unitary. Therefore $(W_{L_1}, S^{\cle_1})$ is a Nica-covariant representation of $O_n$. Clearly,
\[
\clf^2_n \otimes \cle_2= \bigoplus_{\mu \in F^+_n} \bigoplus_{m\geq 0}\left( S_{\mu} \otimes I_{\cle_2}\right) W^m_{L_2}\left(\Omega \otimes \ker L_2^* \right).
\]
Therefore $(W_{L_2}, S^{\cle_2})$ is equivalent to a multiple of the left regular representation of $O_n$, which is a Nica-covariant representation. As $(W_{L_i}, S^{\cle_i})$ for $i=1,2$ are both Nica-covariant representations, therefore $(W_{L}, S^{\cle})$ is also a Nica-covariant representation of $O_n$.
\end{proof}

In the following remark, we provide an alternative proof of the above theorem without using the Wold decomposition.

\begin{remark}
Let $(W_L, S^{\mathcal{E}})$ be a Nica-covariant representation of $O_n$. By Theorem \ref{ W_L iso iff L}, we know that $L$ is an isometry and
\[
L\mathcal{E}\subseteq \mathcal{E}_L\subseteq \overline{span}\{ e_1^{\otimes m} \otimes \eta :m\geq 0, \eta \in \mathcal{E}\}.
\]
We need to show that $L\mathcal{E}\subseteq \Omega \otimes \mathcal{E}$. To this end, fix $m\geq 1$ and $\eta ,\gamma\in \mathcal{E}$. It is enough to show that $\langle L\eta, e_1^{\otimes m}\otimes \gamma\rangle=0$. Indeed, we have
\begin{equation*}
\begin{split}
\langle L\eta, e_1^{\otimes m}\otimes \gamma\rangle & =\langle W_L(\Omega \otimes\eta), e_1^{\otimes m}\otimes \gamma\rangle
\\
& = \langle W_L(\Omega \otimes\eta), (S_1 \otimes I_{\mathcal{E}})^m(\Omega\otimes \gamma)\rangle\\
&= \langle (S_1 \otimes I_{\mathcal{E}})^*W_L(\Omega \otimes\eta), (S_1 \otimes I_{\mathcal{E}})^{m-1}(\Omega\otimes \gamma)\rangle
\\
&= \langle W_L(S_n \otimes I_{\mathcal{E}})^*(\Omega \otimes\eta), (S_1 \otimes I_{\mathcal{E}})^{m-1}(\Omega\otimes \gamma)\rangle
\\
&=0.
\end{split}
\end{equation*}
In the above, we have used the fact that $(W_L, S^{\mathcal{E}})$ is a Nica-covariant representation of $O_n$. Conversely, let $L\mathcal{E}\subseteq \Omega \otimes \mathcal{E}$. We claim that $W_L^*(S_1 \otimes I_{\mathcal{E}})=(S_n \otimes I_{\mathcal{E}})W_L^*$, that is,
\[
(S_1 \otimes I_{\mathcal{E}})^*W_L=W_L(S_n \otimes I_{\mathcal{E}})^*.
\]
It is, however, easy to see that
\[
(S_1 \otimes I_{\mathcal{E}})^*W_L (e_n^{\otimes m} \otimes e_{\mu }\otimes \eta )=W_L(S_n \otimes I_{\mathcal{E}})^*(e_n^{\otimes m} \otimes e_{\mu }\otimes \eta ),
\]
for $m\geq 0, \eta \in \mathcal{E}$ and $\mu=\phi$ or,  $\mu=\mu_1 \dots \mu_k\in \mathcal{F}_n^+$ with $\mu_1\neq n$ (as $L\mathcal{E}\subseteq \Omega \otimes \mathcal{E}$).
\end{remark}

Given an isometry $W_L$, the following theorem highlights the condition on $L$ under which $W_L$ becomes a unitary operator. This subsequently recovers the unitary Fock representations previously obtained by Li in \cite[ Corollary 3.6]{Boyu Li22}. We note that the following assumes $L$ to be an isometry, which is a necessary condition for $W_L$ to be an isometry (see Theorem \ref{ W_L iso iff L}).

\begin{theorem}\label{W_L uni}
Let $\cle$ be a Hilbert space, and let $L \in \clb(\cle, \clf^2_n \otimes \cle)$ be an isometry. Then $W_L$ is unitary if and only if
\[
L \cle = \Omega \otimes \cle.
\]
\end{theorem}
\begin{proof}
Suppose $L \cle = \Omega \otimes \cle$. We intend to apply Theorem \ref{ W_L iso iff L} to the odometer map $W_L$. We already know that $L$ is an isometry. Moreover, in this case, the space $\cle_L$, defined by \eqref{eqn E_L}, simplifies to
\[
\cle_L = \Omega \otimes \cle.
\]
Therefore, by Theorem \ref{ W_L iso iff L}, $W_L$ is an isometry. Now, to show that $W_L$ is onto, it suffices to find the preimage of orthonormal basis elements for $\clf_n^2 \otimes \cle$. Equivalently, it is enough to find preimage of $e_{\mu}\otimes \eta$ for all $\mu \in F_n^+$ and $\eta \in \cle$. This follows at once, looking at the definition of the map $W_L$ and the fact that $L \cle = \Omega \otimes \cle$. However, for the convenience of general readers, we provide complete details. Fix $\eta \in \cle$. As $L \cle = \Omega \otimes \cle,$ there exists $\gamma \in \cle$ such that
\[
L \gamma = \Omega \otimes \eta.
\]
By the definition of $W_L$, we know $W_L(\Omega \otimes \gamma) = \Omega \otimes \eta$, and in general, for all $m \geq 1$, we have
\[
W_L(e_n^{\otimes m} \otimes \gamma)=e_1^{\otimes m} \otimes \eta.
\]
Next, assume that $e_\mu = e_1^{\otimes m} \otimes e_{\mu_1} \otimes \cdots \otimes e_{\mu_k}$, and $\mu_1 > 1$. Again, by the definition of $W_L$, we have that
\[
W_L(e_n^{\otimes m} \otimes e_{\mu_1 - 1} \otimes e_{\mu_2} \otimes \cdots \otimes e_{\mu_k} \otimes \eta) = e_\mu \otimes \eta.
\]
It now follows that $W_L$ is onto and hence a unitary. For the converse, assume that $W_L$ is unitary. In view of Proposition \ref{Lemma: Adjoint of W}, $m=0$ implies that
\[
W_L^*(\Omega \otimes h_l) = \sum_{q \in \Lambda} \overline{c_{l}^{h_q}} (\Omega \otimes h_q) \in \Omega \otimes \cle,
\]
and hence $W_L^* (\Omega \otimes \cle) \subseteq \Omega \otimes \cle$. Now $W_L$ is unitary makes $(W_L, S^\cle)$ a Nica covariant representation of $O_n$, and hence, by Theorem \ref{Thm: Nica covariant characterization}, we have $L \cle \subseteq \Omega \otimes \cle$. By the definition of $W_L$, we now know that $W_L(\Omega \otimes \cle) \subseteq \Omega \otimes \cle$. Therefore, the closed subspace $\Omega \otimes \cle$ reduces $W_L$, and hence
\[
W_L|_{\Omega \otimes \cle}: \Omega \otimes \cle \raro \Omega \otimes \cle,
\]
is a unitary operator. Then $L \cle = W_L(\Omega \otimes \cle) = \Omega \otimes \cle$ proves the necessary part of the theorem.
\end{proof}

Li's proof \cite[Corollary 3.6]{Boyu Li22} of the above theorem employs distinct methodologies, such as Wold decompositions of odometer semigroups.

When $\cle$ is finite dimensional, Theorem \ref{Thm: Nica covariant characterization} simplifies to:
\begin{theorem}\label{Thm: scalar Nica covariant characterization}
Let $\mathcal{E}$  be a finite-dimensional Hilbert space and let $(W_L, S^{\mathcal{E}})$ be an isometric representation of $O_n$. Then, the following statements are equivalent:
  \begin{enumerate}
      \item $(W_L,S^{\mathcal{E}})$ is a Nica-covariant representation of $O_n$.
      \item $L:\mathcal{E}\to \Omega \otimes\mathcal{E}$ is a unitary.
      \item $W_L$ is a unitary.
  \end{enumerate}
  \end{theorem}
  \begin{proof}
\noindent $(1) \Rightarrow (2)$: Let \( (W_L, S^{\mathcal{E}}) \) be a Nica-covariant representation of \( O_n \). Then, Theorem  \ref{ W_L iso iff L} and Theorem \ref{Thm: Nica covariant characterization} together establish that the operator \( L: \mathcal{E} \to \Omega \otimes \mathcal{E} \) is an isometry. Since \( \mathcal{E} \) is finite-dimensional, the rank-nullity theorem further implies that \( L: \mathcal{E} \to \Omega \otimes \mathcal{E} \) is a unitary operator.

\noindent $(2) \Rightarrow (3)$: This is immediate from Theorem \ref{W_L uni}.

\noindent $(3) \Rightarrow (1)$: The fact that \( W_L \) is unitary, together with the relation $W_L (S_n \otimes I_{\cle})= (S_1   \otimes I_{\cle} )  W_L $
, implies that \( (W_L, S^{\mathcal{E}}) \) forms a Nica-covariant representation of \( O_n \).
\end{proof}

Moreover, if we consider the scalar case (that is, $\cle = \C$), then Theorem \ref{Thm: scalar Nica covariant characterization} simplifies to:

\begin{cor} \label{scalar}
The isometric Fock representation \( (W_{\xi}, S) \), where \( \xi \in \mathcal{F}_n^2 \), is Nica-covariant if and only if there exists \( c \in \mathbb{T} \) such that \( \xi = c\Omega \).
\end{cor}

\section{Odometer lifting}\label{sec free rep}

This section aims to draw a connection between noncommutative dilations and representations of $O_n$, which will be referred to as the lifting of representations of $O_n$ or odometer lifting. We begin by discussing the principle of noncommutative isometric dilations of pure row contraction, a theory that is largely attributed to Popescu \cite{Popescue} (also see Bunce \cite{Bunce}, Frazho \cite{DN}, and  Durszt and Sz.-Nagy \cite{Frazho}). Let $\clh$ be a Hilbert space and let $T = (T_1, \ldots, T_n)$ be a pure row contraction on $\clh$; that is, $T$ is a row contraction and satisfies the limit condition in \eqref{eqn pure}. Define the \textit{defect space} $\cld_{T^*}$ as
\[
\cld_{T^*} = \overline{\text{ran}}\Big(I_{\clh} - \sum_{j=1}^{n} T_j T_j^*\Big).
\]
The following is known as the noncommutative dilation of $T$ \cite[Theorem 2.1]{Popescue}. The language is slightly different and adapted to our current context. Let $\cls$ be a closed subspace of a Hilbert space $\clh$. Throughout, we denote $P_\cls$ as the orthogonal projection onto $\cls$.

\begin{theorem}[Dilations of pure row contractions]\label{uni model}
Let $T$ be a pure row contraction on $\clh$. Then there exists an isometry $\Pi_T: \clh \raro \clf^2_n \otimes \cld_{T^*}$ such that
\[
\Pi_T T_i^* = (S_i \otimes I_{\cld_{T^*}})^* \Pi_T,
\]
for all $i=1, \ldots, n$. Moreover, if we set
\[
\clq_{T} = \Pi_T \clh,
\]
then $\clq_{T}$ is a joint invariant subspace of $(S_1^* \otimes I_{\cld_{T^*}}, \ldots, S_n^* \otimes I_{\cld_{T^*}})$, and $T \cong P_{Q_T} S^{D_{T^*}} |_{Q_T}$.

Moreover, $\clf^2_n \otimes \cld_{T^*} = \bigoplus_{\mu \in F_n^+} S_\mu \clq_{T}$.
\end{theorem}

One can write down $\Pi_T$ explicitly. However, for our purposes, we do not need such representations. The final identity is referred to as the minimality property, a distinguished property that makes $\Pi_T$ unique up to unitary equivalence.

\begin{definition}
We refer to the pair $(\Pi_T, \clf^2_n \otimes \cld_{T^*})$ as $T$'s minimal dilation. The closed subspace $\clq_{T}$ will be referred to as a model space.
\end{definition}

In other words, tuples of creation operators on vector-valued Fock spaces are universal models for pure row contractions \cite{Popescue}. This section aims to show that Fock representations are universal models for contractive representations of $O_n$. To give greater clarity, we now introduce the following notion:

\begin{definition}
Let $(W, T)$ defined on $\clh$ be a contractive representation of $O_n$, and let $L \in \clb(\cld_{T^*}, \clf^2_n \otimes \cld_{T^*})$. Then $(W_L, S^{\cld_{T^*}})$ is said to be an odometer lift of $(W, T)$ if
\[
\Pi_T W^* = W_L^* \Pi_T,
\]
where $(\Pi_T, \clf^2_n \otimes \cld_{T^*})$ is the minimal dilation of $T$.
\end{definition}

Suppose $(W_L, S^{\cld_{T^*}})$ is an odometer lift of $(W, T)$ for some symbol $L \in \clb(\cld_{T^*}, \clf^2_n \otimes \cld_{T^*})$. In view of Theorem \ref{uni model}, we have
\[
\Pi_T T_i^* = (S_i \otimes I_{\cld_{T^*}})^* \Pi_T,
\]
for all $i=1, \ldots, n$. Then, with the additional property that $\Pi_T W^* = W_L^* \Pi_T$, it follows that the model space $\clq_{T} ( = \Pi_T \clh)$ is jointly invariant under $(W_L^*, S_1^* \otimes I_{\cld_{T^*}}, \ldots, S_n^* \otimes I_{\cld_{T^*}})$. Therefore, we can say that $(W,T)$ dilates to $(W_L, S^{\cld_{T^*}})$. Also, in particular, $\clq_{T}$ is jointly invariant under $n+1$ tuple $(W_L, S^{\cld_{T^*}})$. With this arrangement, we are now ready to prove our dilation result concerning the representations of $O_n$.

\begin{theorem}\label{thm dil lift}
A contractive representation of $O_n$ always admits an odometer lift.
\end{theorem}

\begin{proof}
Let $(W, T)$ defined on $\clh$ be a contractive representation of $O_n$. First, we establish two essential identities (namely, \eqref{eqn W Sk} and \eqref{eqn W S1} as below) that will play a significant part in the main body of the proof of this theorem. Since $\Pi_T$ is an isometry and $P_{\clq_T} = \Pi_T \Pi_T^*$, it follows that
\[
\Pi_T^* = \Pi_T^* P_{\clq_{T}}.
\]
For each $k<n$, we compute
\[
W \Pi_T^*(S_k \otimes I_{\cld_{T^*}}) = W T_k \Pi_T^* = T_{k+1} \Pi_T^*,
\]	
where the final identity follows from the odometer property that $W T_k = T_{k+1}$. Since $T_{k+1} \Pi_T^* = \Pi_T^* (S_{k+1} \otimes I_{\cld_{T^*}})$, it follows that
\begin{equation}\label{eqn W Sk}
W \Pi_T^*(S_k \otimes I_{\cld_{T^*}}) = \Pi_T^* (S_{k+1} \otimes I_{\cld_{T^*}}),
\end{equation}
for all $k = 1, \ldots, n-1$. This proves the first set of essential identities. For the final one, similar to the above, we compute
\[
W \Pi_T^*(S_n \otimes I_{\cld_{T^*}}) = W T_n \Pi_T^* = T_1 W \Pi_T^* = T_1 \Pi_T^* \Pi_T W \Pi_T^*.
\]
Then $T_1 \Pi_T^* = \Pi_T^*  (S_1 \otimes I_{\cld_{T^*}})$ implies
\begin{equation}\label{eqn W S1}
W \Pi_T^* (S_n \otimes I_{\cld_{T^*}}) = \Pi_T^*  (S_1 \otimes I_{\cld_{T^*}}) \Pi_T W \Pi_T^*.
\end{equation}
Now we turn to construct the required odometer map $W_L$ on $\clf^2_n \otimes \cld_{T^*}$ so that $(W_L, S^{\cld_{T^*}})$ is an odometer lift of $(W,T)$. Define $L: \cld_{T^*} \rightarrow \clf^2_n \otimes \cld_{T^*}$ by
\[
L \eta = (\Pi_T W \Pi_T^*)(\Omega \otimes \eta) \qquad (\eta \in \cld_{T^*}),
\]
where $(\Pi_T, \clf^2_n \otimes \cld_{T^*})$ is the minimal dilation of $T$. Clearly, $L \in \clb(\cld_{T^*}, \clf^2_n \otimes \cld_{T^*})$ is a symbol. Consider the odometer map $W_L$ corresponding to $L$. We claim that $\Pi_T W^*=W_{L}^*\Pi_T$. It is enough to prove this identity on $\{e_\mu\otimes\eta: \mu \in F_n^+, \eta \in \cld_{T^*}\}$. To this end, fix $\eta \in \cld_{T^*}$ and $\mu \in F_n^+$. Suppose $k<n$. By \eqref{eqn W Sk}, it follows that
\[
\begin{split}
W \Pi_T^*(e_k\otimes e_\mu\otimes\eta) & = W \Pi_T^* (S_k \otimes I_{\cld_{T^*}}) (e_\mu\otimes\eta)
\\
& = \Pi_T^* (S_{k+1} \otimes I_{\cld_{T^*}}) (e_\mu\otimes\eta)
\\
& = \Pi_T^* (e_{k+1} \otimes e_\mu\otimes\eta)
\\
& = \Pi_T^*  W_L(e_k \otimes e_\mu\otimes\eta).
\end{split}
\]
Now we need to focus on vectors of the form $e_n \otimes e_\mu \otimes \eta$. First, we consider the case of $e_n \otimes \eta$. In this case, we use the identity \eqref{eqn W S1} to derive
\[
W \Pi_T^*(e_n\otimes \eta) = W \Pi_T^* (S_n \otimes I_{\cld_{T^*}}) (\Omega \otimes \eta) = \Pi_T^* (S_1 \otimes I_{\cld_{T^*}}) \Pi_T W \Pi_T^*(\Omega\otimes\eta).
\]
But $(\Pi_T W \Pi_T^*)(\Omega \otimes \eta) = L \eta$, and hence
\[
\Pi_T^* (S_1 \otimes I_{\cld_{T^*}}) \Pi_T W \Pi_T^*(\Omega\otimes\eta) = \Pi_T^*(e_1 \otimes L \eta).
\]
Since, $W_L(e_n\otimes \eta) = e_1 \otimes L \eta$, by the definition of the odometer map $W_L$, it follows that
\[
W \Pi_T^*(e_n\otimes \eta) = \Pi_T^* W_L(e_n\otimes \eta).
\]
Finally, we consider the general case of $e_n \otimes e_\mu \otimes \eta$. Again, using the identity \eqref{eqn W S1}, we see
\[
\begin{split}
W \Pi_T^*(e_n\otimes e_\mu \otimes\eta) & = W \Pi^*_T (S_n \otimes I_{\cld_{T^*}}) (e_\mu\otimes\eta)
\\
& = \Pi_T^* (S_1\otimes I_{\cld_{T^*}}) \Pi_T W \Pi_T^* (e_\mu\otimes\eta)
\\
& = \Pi_T^* (e_1 \otimes \Pi_T W \Pi^*_T (e_\mu\otimes\eta)).
\end{split}
\]
Using the preceding two cases repeatedly along with the identities $\Pi_T^* = \Pi_T^* P_{\clq_T}$ and $P_{\clq_T} (S_k \otimes I_{\cld_{T^*}}) = P_{\clq_T} (S_k \otimes I_{\cld_{T^*}})P_{\clq_T}$ for all $k$, we conclude that
\[
W \Pi_T^*(e_n\otimes e_\mu \otimes\eta) =  \Pi_T^* W_L (e_n \otimes e_{\mu} \otimes \eta).
\]
This completes the proof of the theorem.
\end{proof}

From the standpoint of model spaces, Theorem \ref{thm dil lift} and the discussion that came before it yield the following:

\begin{cor}\label{cor model to odometer}
Let $(W,T)$ be a contractive representation of $O_n$. Then there exist a symbol $L \in \clb(\cld_{T^*}, \clf^2_n \otimes \cld_{T^*})$ and a closed subspace $\clq_{T} \subseteq \clf^2_n \otimes \cld_{T^*}$ such that $(S_j \otimes I_{\cld_{T^*}})^* \clq_{T} \subseteq \clq_{T}$ for all $j=1, \ldots, n$, and $W_L^* \clq_{T} \subseteq \clq_{T}$, and
\[
(W, T) \cong P_{Q_T} (W_L, S^{D_{T^*}}) |_{Q_T}.
\]
\end{cor}

This result is, on the one hand, dilation for contractive representations of $O_n$ in the sense of Sz.-Nagy and Foias \cite{NagFoi10}, and Popescu \cite{Popescue char}. On the other hand, it shares some (but not all) resemblance to Sarason's commutant lifting theorem \cite{Sarason} (and also to the noncommutative commutant lifting theorem). More specifically, in the present scenario, we are lifting the odometer maps acting on Hilbert spaces to the odometer maps acting on the Fock spaces (the minimal isometric dilation spaces). Let us see the relation between $\norm{W}$ and $\norm{W_L}$.
\noindent Let $(W, T)$ be a contractive representation of $O_n$. By Theorem \ref{thm dil lift} and its proof, we know that $(W_L, S^{\cld_{T^*}})$ is an odometer lift of $(W, T)$, where the symbol $L \in \clb(\cld_{T^*}, \clf^2_n \otimes \cld_{T^*})$ is given by
\[
L \eta = \Pi_{T} W \Pi^*_{T}(\Omega \otimes \eta),
\]
for all $\eta \in \cld_{T^*}$. As $W = \Pi_T^* W_L \Pi_T$ and $\Pi_T$ is an isometry, it follows that $\norm {W} \leq \norm{W_L}$. On the other hand, since $\norm{L} = \sup \{\norm {L \eta}: \eta \in \cle, \norm{\eta}=1\}$, we compute
\begin{align*}
\norm{W_L} & \leq 1+ \norm{L}
\\
& = 1+ \sup \{ \norm {\Pi_{T} W \Pi^*_{T}(\Omega \otimes \eta)}: \eta \in \cle, \norm{\eta}=1\}
\\
& \leq 1+ \norm{\Pi_{T}}\norm {W }\norm{\Pi^*_{T}}  \sup \{\|\Omega \otimes \eta\|: \eta \in \cle, \norm{\eta}=1\}
\\
& = 1+ \norm {W},
\end{align*}
as $\|\Pi_T\| = 1$. This implies that
\[
\norm{W} \leq  \norm {W_L} \leq 1+ \norm{W}.
\]

\section{Subrepresentations}\label{sec subrepr}

Now we turn to joint invariant subspaces of Fock representations, which will also be referred to as subrepresentations of the odometer semigroup. First, we introduce the concept of invariant subspaces in the traditional context and then extend it to invariant subspaces of Fock representations (or subrepresentations of $O_n$) in the present context.

\begin{definition}
Let $(W_L, S^\cle)$ be a Fock representation. A closed subspace $\cls \subseteq \clf^2_n \otimes \cle$ is referred to as an invariant subspace if
\[
(S_j \otimes I_\cle) \cls \subseteq \cls,
\]
for all $j=1, \ldots, n$. If, in addition, $W_L \cls \subseteq \cls$, then we call $\cls$ \textit{invariant under the Fock representation} $(W_L, S^\cle)$. We also say in this instance that $(W_L|_{\cls}, S^{\cle}|_{\cls})$ is a subrepresentation of the Fock representation $(W_L, S^\cle)$.
\end{definition}

In the definition above, the symbol $S^{\cle}|_{\cls}$ denotes the $n$-tuple on $\cls$, which is defined as follows:
\[
S^{\cle}|_{\cls} = ((S_1 \otimes I_\cle)|_{\cls}, \cdots, (S_n \otimes I_\cle)|_\cls).
\]
Recall that a bounded linear operator $\Phi: \clf^2_n \otimes \cle_* \raro \clf^2_n \otimes \cle$ is \textit{multi-analytic} \cite{Popescu 95} if
\begin{equation}\label{eqn multi analytic}
\Phi (S_j \otimes I_{\cle_*}) = (S_j \otimes I_{\cle}) \Phi,
\end{equation}
for all $j=1, \ldots, n$. Moreover, a multi-analytic operator $\Phi$, as above, is \textit{inner} if $\Phi$ is an isometry. Popescu's result \cite{Popescue char}, commonly referred to as the noncommutative Beurling-Lax-Halmos theorem, links up invariant subspaces with multi-analytic inner functions. To facilitate the subsequent construction, we will now offer a sketch of the proof of the result. The technique is standard and simply follows the lines of \cite[Theorem 2]{Frazho 84} and \cite[Theorem 2.2]{Popescue char}. Here the only distinction is in our emphasis on the factorizations of the multi-analytic operators, which will yield fresh perspectives that we will elaborate on shortly.

\begin{theorem}\label{Popescue multi-analytic}
Let $\cls \subseteq \clf^2_n \otimes \cle$ be a closed subspace. Then $\cls$ is an invariant subspace if and only if there exist a Hilbert space $\cle_*$ and an inner multi-analytic operator $\Phi: \clf^2_n \otimes \cle_* \raro \clf^2_n \otimes \cle$ such that
\[
\cls = \Phi(\clf^2_n \otimes \cle_*).
\]
Moreover, there exists a unitary $\Pi: \clf^2_n \otimes \cle_* \raro \cls$ such that $\Pi (S_i \otimes I_{\cle_*}) = (S_i \otimes I_{\cle}) \Pi$ for all $i=1, \ldots, n$, and
\[
\Phi = i_\cls \circ \Pi.
\]
\end{theorem}
\begin{proof}
If $\cls = \Phi(\clf^2_n \otimes \cle_*)$, then \eqref{eqn multi analytic} immediately implies that $\cls$ is an invariant subspace. For the reverse direction, for each $j=1, \ldots, n$, set $V_j = (S_j \otimes I_\cle)|_{\cls}$. Since $\cls$ is an invariant subspace, for all $i,j \in \{1, \ldots, n\}$, it follows that $V_i^* V_j = \delta_{ij}I_\cls$. As $\mathcal{S}$ is invariant under $S^{\cle}$, \cite[Theorem 2.1]{DP} implies that $\mathcal{S}$ has wandering subspace, say $\cle_*.$ Therefore $(V_1, \ldots, V_n)$ is pure row isometry on $\cls$, and we have the following orthogonal decomposition:
\[
\cls = \bigoplus_{\mu \in F_n^+} V_{\mu} \cle_* = \bigoplus_{\mu \in F_n^+} (S_\mu \otimes I_\cle) \cle_*,
\]
where
\[
\cle_* = \cls \ominus \sum_{i=1}^{n} (S_i \otimes I_\cle) \cls.
\]
Then
\[
\Pi(e_\mu \otimes \eta) = (S_\mu \otimes I_\cle)\eta \qquad(\mu \in F_n^+, \eta \in \cle_*),
\]
defines a unitary $\Pi: \clf^2_n \otimes \cle_* \raro \cls$. Moreover, it is easy to see that $\Pi (S_i \otimes I_{\cle_*}) = (S_i \otimes I_{\cle}) \Pi$ for all $i=1, \ldots, n$. Now we consider the inclusion map $i_\cls :\cls \hookrightarrow \clf^2_n \otimes \cle$, and define
\[
\Pi_\cls = i_\cls \circ \Pi \in \clb(\clf^2_n \otimes \cle_*, \clf^2_n \otimes \cle).
\]
Since $i_\cls$ is an isometry, it is evident that $\Pi_\cls$ is an isometry and $\Pi_\cls \Pi_\cls^* = P_\cls$. An easy computation shows that $\Pi_\cls (S_i \otimes I_{\cle_*}) = (S_i \otimes I_{\cle}) \Pi_\cls$ for all $i=1, \ldots, n$, and consequently, there exists an inner multi-analytic operator $\Phi \in \clb(\clf^2_n \otimes \cle_*, \clf^2_n \otimes \cle)$ such that
\[
\Phi = \Pi_\cls = i_\cls \circ \Pi \in \clb(\clf^2_n \otimes \cle_*, \clf^2_n \otimes \cle),
\]
and $\cls = \Phi(\clf^2_n \otimes \cle_*)$.
\end{proof}

The operator $\Phi$ is unique in the following sense: If $\cls = \Phi(\clf^2_n \otimes \cle_*)$ is an invariant subspace of $\clf^2_n \otimes \cle$ as above, and if $\cls = \tilde \Phi(\clf^2_n \otimes \tilde\cle_*)$ for some Hilbert space $\tilde\cle_*$ and inner multi-analytic operator $\tilde \Phi: \clf^2_n \otimes \tilde \cle_* \raro \clf^2_n \otimes \cle$, then there exists a unitary operator $\tau \in \mathcal{B}(\tilde{\cle_*}, \cle_*)$ such that
\[
\tilde{\Phi} = \Phi (I_{\clf^2_n} \otimes \tau).
\]
We again stress the factorization $\Phi = i_\cls \circ \Pi$ in the above theorem, which, in terms of the commutative diagram, yields the following: 

\setlength{\unitlength}{3mm}
\begin{center}
\begin{picture}(40,16)(0,0)
\put(15.7,3){$\cls$}\put(20,2.3){$i_\cls$}
\put(23.9,3){$\clf^2_n \otimes \cle$} \put(22, 10){$\clf^2_n \otimes \cle_*$} \put(24,9.2){ \vector(0,-1){5}} \put(22,
9.4){\vector(-1,-1){5.3}} \put(17,
3.4){\vector(1,0){5.9}}\put(18,7.2){$\Pi$}\put(24.7,6.5){$\Phi$}
\end{picture}
\end{center}
The way we factorize $\Phi$ will be significant in what comes next. And this is where our use of the noncommutative Beurling-Lax-Halmos theorem will differ from how it was used earlier. For instance:

\begin{lemma} \label{action of M_theta}
In the setting of Theorem \ref{Popescue multi-analytic}, suppose $\cls = \Phi(\clf^2_n \otimes \cle_*)$. For each $\mu \in F_n^+$ and $\eta \in \cle_*$, we have the following:
\begin{enumerate}
\item $\Phi(e_{\mu} \otimes \eta) = (S_\mu \otimes I_{\cle}) \eta$.
\item $\Phi^*(S_\mu \otimes I_{\cle}) \eta = e_{\mu} \otimes \eta$.
\end{enumerate}
\end{lemma}
\begin{proof}
By the factorization part of Theorem \ref{Popescue multi-analytic}, we know that $\Phi = i_\cls \circ \Pi$. Since $\cle_* = \cls \ominus \sum_{i=1}^{n} (S_i \otimes I_\cle) \cls
$, we have $\eta \in \cle_* \subseteq \cls$, so that
\[
(S_\mu \otimes I_{\cle}) \eta \in \cls,
\]
as $\cls$ is an invariant subspace. Then
\[
\Phi(e_{\mu} \otimes \eta) = i_\cls \circ \Pi (e_{\mu} \otimes \eta) = i_\cls (S_\mu \otimes I_{\cle}) \eta = (S_\mu \otimes I_{\cle}) \eta.
\]
The second identity can be easily derived from the fact that $\Phi$ is isometric.
\end{proof}

Now we are ready for the invariant subspaces of Fock representations. However, we must lay the right foundation. Let $\cls \subseteq \clf^2_n \otimes \cle$ be a closed subspace. First, assume that $\cls$ is an invariant subspace. By Theorem \ref{Popescue multi-analytic}, there exist a Hilbert space $\cle_*$ and an inner multi-analytic operator $\Phi \in \clb(\clf^2_n \otimes \cle_*, \clf^2_n \otimes \cle)$ such that
\[
\cls = \Phi (\clf^2_n \otimes \cle_*).
\]
Next, consider a Fock representation $(W_L, S^\cle)$. According to Douglas' range inclusion theorem \cite{Douglas}, $\cls$ is invariant under $(W_L, S^\cle)$ if and only if there is $C \in \clb(\clf^2_n \otimes \cle_*)$ such that
\[
W_L \Phi = \Phi C.
\]
If such a $C$ exists, it would be unique, and this simply follows from the fact that
\begin{equation}\label{eqn C}
C = \Phi^* W_L \Phi.
\end{equation}
The remaining task, therefore, along with its existence, is to determine the representation of the map $C$. This is what we do in the following theorem, which in particular asserts that $C$ is an odometer map and subsequently classifies invariant subspaces of Fock representations. 

\begin{theorem} \label{invariant subspace}
Let $\cle$ be a Hilbert space, $(W_L, S^\cle)$ be a Fock representation, and let $\cls \subseteq \clf^2_n \otimes \cle$ be a closed subspace. Then $\cls$ is invariant under $(W_L, S^\cle)$ if and only if there exist a Hilbert space $\cle_*$, an inner multi-analytic operator $\Phi \in \clb(\clf^2_n \otimes \cle_*, \clf^2_n \otimes \cle)$, and a Fock representation $(W_{L_*}, S^{\cle_*})$ such that
\[
\cls = \Phi (\clf^2_n \otimes \cle_*),
\]
and
\[
W_L \Phi = \Phi W_{L_*}.
\]
\end{theorem}
\begin{proof}
Following Theorem \ref{Popescue multi-analytic} and what we have discussed preceding the statement of this theorem, all that is left is to show that there is an $L_* \in \clb(\cle_*, \clf^2_n \otimes \cle_*)$ such that $W_L \Phi = \Phi W_{L_*}$, where $\cls$ is the invariant subspace of $\clf^2_n \otimes \cle$ admitting inner multi-analytic representation $\cls = \Phi (\clf^2_n \otimes \cle_*)$ and $W_L \cls \subseteq \cls$. By \eqref{eqn C}, we already know that there is a $C \in \clb(\clf^2_n \otimes \cle_*)$ such that
\[
C = \Phi^* W_L \Phi.
\]
We need to prove that $C = W_{L_*}$ for some $L_* \in \clb(\cle_*, \clf^2_n \otimes \cle_*)$. We now start to build the map $L_*$. It is crucial to recall from the proof of Theorem \ref{Popescue multi-analytic} that
\[
\cle_* = \cls \ominus \sum_{i=1}^{n} (S_i \otimes I_\cle) \cls.
\]
Recall also that the inner multi-analytic operator $\Phi$ factors as (see \eqref{eqn C})
\[
\Phi = i_\cls \circ \Pi \in \clb(\clf^2_n \otimes \cle_*, \clf^2_n \otimes \cle),
\]
where $i_\cls :\cls \hookrightarrow \clf^2_n \otimes \cle$ is the inclusion map, and $\Pi: \clf^2_n \otimes \cle_* \raro \cls$ is the unitary operator defined by
\[
\Pi(e_\mu \otimes \eta) = (S_\mu \otimes I_\cle)\eta,
\]
for all $\mu \in F_n^+$ and $\eta \in \cle_*$. Also recall that $W_L \in \clb(\clf^2_n \otimes \cle)$. As
\[
\cle_* \subseteq \cls \subseteq \clf^2_n \otimes \cle,
\]
it follows that $W_L|_{\cle_*} : \cle_* \raro \clf^2_n \otimes \cle$ is a well-defined bounded linear operator. Moreover, since $W_L \cls \subseteq \cls$, we conclude that
\[
W_L|_{\cle_*}: \cle_* \raro \cls.
\]
Finally, we define the symbol $L_* \in \clb(\cle_*, \clf^2_n \otimes \cle_*)$ by
\[
L_* = \Pi^* W_L|_{\cle_*}.
\]
Therefore, $(W_{L_*}, S^{\cle_*})$ on $\clf^2_n \otimes \cle_*$ is a Fock representation. It is now enough to prove that
\[
\Phi^* W_L \Phi = W_{L_*}.
\]
To this end, fix $\eta \in \cle_*$ and $\mu \in F_n^+$. Assume that $k < n$. By applying Lemma \ref{action of M_theta} twice, we obtain
\[
\begin{split}
\Phi^* W_L \Phi(e_{k}\otimes e_{\mu} \otimes \eta) & = \Phi^*  W_{L} (S_{k} \otimes I_{\cle_*}) ( S_{\mu} \otimes I_{\cle_*}) \eta
\\
& = \Phi^* (S_{k+1} \otimes I_{\cle_*}) (S_\mu  \otimes I_{\cle_*}) \eta
\\
& = e_{k+1}\otimes  e_\mu  \otimes \eta.
\end{split}
\]
Moreover, since $W_{L_*} (e_{k}\otimes e_{\mu} \otimes \eta) = e_{k+1}\otimes  e_\mu  \otimes \eta$, we conclude that
\[
\Phi^* W_L \Phi(e_{k}\otimes e_{\mu} \otimes \eta) = W_{L_*} (e_{k}\otimes e_{\mu} \otimes \eta),
\]
for all $k <n$. Next, consider the basis element $e_n \otimes \eta$. Again, Lemma \ref{action of M_theta} implies
\[
\Phi^* W_L \Phi \left(e_{n} \otimes \eta\right)	= \Phi^* W_L (S_n \otimes I_{\cle}) \eta = \Phi^* (S_1 \otimes I_{\cle}) W_{L} (\eta).
\]
As $\eta \in \cle_* \subseteq \cls$ and $W_{L} \cls \subseteq \cls$, it follows that ${W_L} (\eta) \in \cls$. By using the orthogonal direct sum decomposition of $\cls$, we write
\[
W_{L}(\eta)=\sum_{\alpha} c_{\alpha} (S_{\alpha} \otimes I_{\cle}) \eta_{\alpha},
\]
where $\eta_{\alpha} \in \cle_{*} \subseteq \cls$ for all $\alpha \in F_n^+$. Then $\Pi^* W_L (\eta)=\sum_{\alpha} c_{\alpha} e_\alpha\otimes \eta_{\alpha}$, and consequently
\[
\begin{split}
\Phi^* (S_1 \otimes I_{\cle}) W_{L} (\eta) & = \Phi^* (S_1 \otimes I_{\cle}) \Big(\sum_{\alpha} c_{\alpha} (S_\alpha \otimes I_{\cle})  \eta_{\alpha} \Big)
\\
& =  e_{1} \otimes \sum_{\alpha} c_{\alpha} e_\alpha\otimes \eta_{\alpha}
\\
& = e_{1} \otimes \Pi^* W_L(\eta),
\end{split}
\]
which, along with $ e_{1} \otimes \Pi^* W_L(\eta) = W_{L_*}(e_n \otimes \eta)$ implies that
\[
{\Phi}^* W_L \Phi (e_n \otimes \eta) = W_{L_*}(e_n \otimes \eta).
\]
Finally, we consider the basis vector $e_n \otimes e_\mu \otimes \eta \in \clf^2_n \otimes \eta_*$. Lemma \ref{action of M_theta} again implies that
\[
\begin{split}
\Phi^{*} W_{L} \Phi(e_{n} \otimes e_{\mu} \otimes \eta) = \Phi^* W_{L}(S_n  \otimes I_{\cle}) (S_\mu  \otimes I_{\cle}) \eta = \Phi^{*}(S_1  \otimes I_\cle) W_L (S_\mu  \otimes I_{\cle}) \eta.
\end{split}
\]
Depending on the symbol $\mu$, the definition of odometer relations and Lemma \ref{action of M_theta} implies that
\[
\Phi^{*} W_{L} \Phi(e_{n} \otimes e_{\mu} \otimes \eta) = W_{L_*}(e_n \otimes e_{\mu} \otimes \eta).
\]
This says $\Phi^* W_L \Phi$ and $W_{L_*}$ agree on the orthonormal basis vectors, and hence $\Phi^* W_L \Phi = W_{L_*}$, completing the proof of the theorem.
\end{proof}

The formulation of the problem of invariant subspaces of Fock representations is comparable to a similar problem in a different context, which is described in [2] in answer to a question by J. Agler and N. Young.

The following corollary is now easy and follows the lines of the classical Beurling theorem:

\begin{cor}\label{cor Unit equiv Fock rep}
Let $\cls$ be an invariant subspace of a Fock representation $(W_L, S^\cle)$. Then there exists a Fock representation $(W_{L_*}, S^{\cle_*})$ such that
\[
(W_L|_{\cls}, S^\cle|_{\cls}) \cong (W_{L_*}, S^{\cle_*}).
\]
\end{cor}
\begin{proof}
We know that there is a Hilbert space $\cle_*$, an inner multi-analytic operator $\Phi \in \clb(\clf^2_n \otimes \cle_*, \clf^2_n \otimes \cle)$, and a symbol $L_* \in \clb(\cle_*, \clf^2_n \otimes \cle_*)$ such that $\cls = \Phi (\clf^2_n \otimes \cle_*)$ and $W_L \Phi = \Phi W_{L_*}$. Then
\[
U (f \otimes \eta) = \Phi(f \otimes \eta) \qquad (f \in \clf^2_n, \eta \in \cle_*),
\]
defines a unitary $U: \clf^2_n \otimes \cle_* \raro \cls$. Evidently, $U W_{L_*} = W_L|_{\cls} U$ and $U (S_i \otimes I_{\cle_*}) = (S_i \otimes I_{\cle})|_{\cls} U$ for all $i=1, \ldots, n$. This completes the proof of the corollary.
\end{proof}

Therefore, up to unitary equivalence, subrepresentations of Fock representations are also Fock representations.

\section{Examples}\label{sec examples}

In this concluding section, we aim to present some examples of representations of $O_n$ and compute the spectrum of unitary odometer operators.

\subsection{Nica covariant} We begin with a simple example of Nica covariant representation of $O_n$.

\begin{example}
Let $\mathcal{H}$ be Hilbert space, $\{e_{n}: n\in \Z_+\}$ be an orthonormal basis for $\clh$, and let $q \in \T$. Define isometries $V_{1}$, $V_{2}$ and $W$ on $\clh$ as follows:
\[
V_1(e_{k}) = \overline{q}^{2k}e_{2k}, \; V_2(e_{k}) = \overline{q}^{2k+1}e_{2k+1},
\]
and
\[
W(e_{k}) = \overline{q}e_{k+1},
\]
for all $k \in \Z_+$. It is easy to check that $V=(V_{1}, V_{2})$ is a row isometry. Moreover, $WV_{1}=V_{2}$ and $W^{*}V_{1}=\overline{q} V_{2}W^*$, which also implies that $WV_{2}=q V_{1}W$. If $q=1$, then $(W,V)$ becomes a Nica covariant representation of $O_n$.
\end{example}

Next, we present an example of an isometric Fock representation, which is also a weak bi-shift. In other words, in the following, given a Hilbert space $\cle$, we will provide an explicit example of a map $L \in \clb(\cle, \clf^2_n \otimes \cle)$ such that $W_L$ is an isometry on $\clf^2_n \otimes \cle$. In the following, $\Lambda$ is either a finite set or $\Z_+$.

\begin{example}\label{example vector Fock rep}
Let $\cle$ be a Hilbert space, and let $\{h_p\}_{p \in \Lambda}$ be an orthonormal basis for $\cle$. For each $m\in \Lambda$, define
\[
\mathcal{E}_m=\text{span}\{h_m\},
\]
and
\[
K_m= \overline{\text{span~}} \{e_{\mu} \otimes h_m: \mu \in F^+_n\}.
\]
Then
\begin{equation*}
    \begin{split}
        \mathcal{F}_n^2 \otimes \mathcal{E} &=\bigoplus_{m\in \Lambda}\bigoplus_{\mu \in \mathcal{F}_n^+} K_m.
    \end{split}
\end{equation*}
Define $L_m:\cle_m\to K_m$ by
$$
L_m(\lambda h_m)=\lambda e_1^{\otimes m}\otimes h_m,
$$
for all $h_m \in \cle_m$ and $\lambda\in \C$. Each $L_m$ is clearly an isometry and
\[
\text{ran} L_m \subseteq \mathcal{E}_{L_m}.
\]
Therefore, by Theorem \ref{ W_L iso iff L}, $(W_{L_m},S^{\mathcal{E}_m})$ is an isometric
representation of $O_n$ on $K_m$. Now, we have $W_{L_m} (S_n \otimes I_{\cle})^*(\Omega \otimes h_m)=0$ and
\[
(S_1 \otimes I_{\cle})^* W_{L_m} (\Omega \otimes h_m)=(S_1 \otimes I_{\cle})^*(e_1^{\otimes m} \otimes h_m)= e_1^{\otimes (m-1)} \otimes h_m,
\]
for all $h_m \in \cle_m$. That is, $(W_{L_m}, S^{\cle_m})$ is not a Nica-covariant representation. By \cite[Theorem 3.10]{Boyu Li22}, we write, $K_m= {\clh}^m_{us} \oplus {\clh}^m_{ws}$. We know that
$$
{\clh}^m_{us} = \bigoplus_{\mu \in F^+_n} \left( S_{\mu} \otimes I_{\cle_m}\right) \left( \bigcap_{i=1}^{n} \ker (S_i \otimes I_{\cle_m})^* \cap {\clh}^{W_m}_{u} \right).
$$
Clearly,
\[
\bigcap_{i=1}^{n} \ker (S_i \otimes I_{\cle_m})^*=\Omega \otimes \cle_m \subseteq \ker W^*_m,
\]
which implies $\Omega \otimes \cle_m \cap {\clh}^{W_m}_{u} =0$. Therefore, $\clh^m_{us}=0,$ and  $\clh^m_{ws}=K_m$. Consequently, $(W_{L_m}, S^{\cle_m})$ is a weak bi-shift. Finally, set
\[
L:=\bigoplus_{m\in \Lambda} L_m.
\]
Then $(W_L, S^{\cle})$ is a weak bi-shift, where
\[
W_L= \bigoplus_{m\in \Lambda} W_{L_m} \in \mathcal{B}(\mathcal{F}_n^2\otimes \mathcal{E}),
\]
and
\[
S^{\mathcal{E}}= \bigoplus_{m\in \Lambda} S^{{\mathcal{E}}_m}.
\]
Also as each \( (W_{L_m}, S^{\mathcal{E}_m}) \) is an isometric representation, it follows that \( (W_L, S^{\cle}) \) is also an isometric representation of \( O_n \).
\end{example}

The example above suggests that Li's construction \cite{Boyu Li22} cannot be extended beyond unitary Fock representations of $O_n$. In the latter scenario, the symbol represents a mapping from $\cle$ to itself.

In the following example, we highlight, in the context of Theorem \ref{Thm: Nica covariant characterization}, that Nica covariant representations $(W_L, S^\cle)$ exist such that $W_L$ is isometric but not unitary:

\begin{example}
Let $\cle$ be a Hilbert space and let $\{h_p\}_{p \in \Z_+}$ be an orthonormal basis for $\cle$. Define $L \in \mathcal{B}(\cle,\clf^2_n \otimes \cle)$ by
\[
L(h_p) = \Omega \otimes h_{p+1} \qquad (p \in \Z_+).
\]
Clearly, $L$ is an isometry (in fact, it is a kind of shift operator). It is easy to see that $(W_L, S^\cle)$ is a Nica covariant representation. However, $W_L$ is not a unitary. Also,  applying  similar arguments as in the converse of Theorem \ref{Thm: Nica covariant characterization}, it follows that $(W_L, S^{\mathcal{E}})$ is the left regular representation of the semigroup $O_n$.
 \end{example}

\subsection{Fock represenations} It makes sense to wonder if there are examples of nontrivial sequences that meet the conditions of Corollary \ref{cor scalar iso Fock rep}. In order to address this, we first consider a sequence of scalars $\{c_{p}: p\in \mathbb{Z}_{+}\}$ that satisfies the conditions of Corollary \ref{cor scalar iso Fock rep}. Assuming $\{c_{p}: p\in \mathbb{Z}_{+}\}$ is a finite sequence of nonzero numbers, by an obvious deduction, there exists $m \in \Z_+$ such that
\[
c_p = 0 \qquad(p \neq m).
\]
Therefore, we need to focus on nontrivial sequences. With this view in mind, now we illustrate Corollary \ref{cor scalar iso Fock rep} with a concrete example.

\begin{example}\label{example sequence}
Consider the quadratic equation
\[
x^2 - x - 1 = 0.
\]
Clearly
\[
\omega= \frac{1-\sqrt{5}}{2},
\]
is a solution to the above equation. Construct a sequence $\{c_{p}\}_{p\in \mathbb{Z}_{+}}$ by defining
\[
c_p =
\begin{cases}
\sqrt{ \frac{2}{\sqrt{5}+3}} & \mbox{if } p = 0 \\
\sqrt{ \frac{2}{\sqrt{5}+3}}\omega^{p-1} & \mbox{if } p \geq 1.
\end{cases}
\]
First, we compute
\begin{align*}
\sum_{p=0}^{\infty}\vert c_{p}\vert^2 &= \frac{2}{\sqrt{5}+3} \Big(1+ \sum_{p=1}^{\infty}\omega^{2(p-1)} \Big)
\\
&=\frac{2}{\sqrt{5}+3} \Big(1+ \frac{1}{1-\omega^2}\Big)
\\
&=\frac{2}{\sqrt{5}+3} \Big( \frac{2- \omega^2}{1-\omega^2}\Big)
\\
&=\frac{2}{\sqrt{5}+3} \Big( \frac{1- \omega}{-\omega}\Big) \qquad (\text{as~} 1+ \omega- \omega^2=0)
\\
&=\frac{2}{\sqrt{5}+3} \Big( \frac{\sqrt{5}+1}{\sqrt{5}-1}\Big)
\\
&=1.
\end{align*}
Moreover, for each $r\geq 1,$  we have
\begin{align*}
\sum_{p=0}^{\infty} c_{p+r}\overline{c_{p}}& = \frac{2}{\sqrt{5}+3} \Big( \omega^{r-1} + \sum_{p=1}^{\infty} \omega^{p + r-1} \omega^{p-1} \Big)
\\
&= \frac{2}{\sqrt{5}+3} \Big( \omega^{r-1}+\sum_{p=1}^{\infty} \omega^{2p+r-2} \Big)
\\
&= \frac{2}{\sqrt{5}+3} \Big( \omega^{r-1}+ \omega^{r}\sum_{p=1}^{\infty} \omega^{2p-2} \Big)\\
&= \frac{2}{\sqrt{5}+3} \Big( \omega^{r-1}+  \frac{\omega^{r}}{1- \omega^2} \Big)\\
&= \frac{2}{\sqrt{5}+3}  \omega^{r-1}\Big(  \frac{1- \omega^2+\omega}{1- \omega^2} \Big).
\end{align*}
Since $\omega^2 - \omega - 1=0$, we conclude that
\[
\sum_{p=0}^{\infty} c_{p+r}\overline{c_{p}} = 0,
\]
for all $r \geq 1$. This is condition \eqref{eqn E_L}. If we define
\[
\xi=\sum^{\infty}_{p=0} c_{p}e^{\otimes p}_{1},
\]
then $W_{\xi}$ is an isometry, as the sequence $\{c_{p}\}_{p\in \mathbb{Z}_{+}}$ meets all of Corollary \ref{cor scalar iso Fock rep}'s specifications. Hence there are nontrivial examples of isometric Fock representations of $O_n$.
\end{example}

The preceding example indicates once more that the symbol of non-unitary Fock representations of $O_n$, even at the level of the scalar case, does not have to be a self-map, unlike how it appeared in Li's construction \cite{Boyu Li22} of unitary Fock representations of $O_n$.
	
\subsection {$\sigma(W)$} Theorem \ref{W_L uni} gives a necessary and sufficient condition for $W_L$ to become unitary. This motivates our interest in determining the spectrum of the operator $W_L$.

\begin{prop}\label{prop: spect}
Suppose $(W,S^{\mathcal{E}})$ is a unitary Fock representation of $O_n$. Then $\sigma(W)=\T$.
\end{prop}
\begin{proof}
Since $(W,S^{\mathcal{E}})$ is a unitary Fock representation of $O_n,$ by Theorem \ref{Prop: Odometer operator} and Theorem \ref{W_L uni}, there is a unique unitary map $L \in \clb(\cle,  \Omega \otimes \cle)$ such that
$
W = W_L.
$
Let us decompose
$$\mathcal{F}_{n}^{2}\otimes \mathcal{E}= \bigoplus_{m=0}^{\infty} (\mathcal{H}_{m}\otimes\mathcal{E}),$$
where $\mathcal{H}_{m}= {\rm span}\{e_{\mu}: \vert\mu\vert= m \}$. Note that $W_{L}|_{\mathcal{H}_{m}\otimes \mathcal{E}}$ is a unitary operator and
$$
(W_{L}|_{\mathcal{H}_{m}\otimes \mathcal{E}})^{n^m}=I_{\mathcal{H}_{m}}\otimes L.
$$
Now, we compute
\begin{align*}
          \sigma(W_{L}) &= \overline{\bigcup^{\infty}_{m=0}\sigma(W_{L}|_{\mathcal{H}_{m}\otimes \mathcal{E}})} \\
          &= \overline{\bigcup^{\infty}_{m=0}\{\lambda\in \mathbb{C}: \lambda^{n^m}\in \sigma(I_{\mathcal{H}_{m}}\otimes L)\}}\\
       &= \overline{\bigcup^{\infty}_{m=0}\{\lambda\in \mathbb{C}: \lambda^{n^m}\in \sigma(L)\}}\\
             &\supseteq \overline{\bigcup^{\infty}_{m=0}\{\lambda\in \mathbb{C}: \lambda^{n^m}= c , \text{ for some }
             \vert c\vert=1\}} ~~~\,\,\,(\text{ as } L \text{ is unitary })\\
             &=  \T.
      \end{align*}
 Since ${W}_{L}$ is unitary, $\sigma(W_{L})= \T.$
\end{proof}

In closing, we again reiterate Remark \ref{remark odometer maps} that the concept of odometer maps relies on symbols, which can vary all over the Fock space. This is as opposed to the noncommutative Toeplitz operators, which are by themselves a fascinating subject of research. This calls for an in-depth investigation of odometer maps acting on vector-valued Fock spaces. Moreover, with regard to the structure of $C^*$-algebras generated by Fock representations of $O_n$, we expect that the results of this paper; more precisely, the explicit description of the isometric Fock representations; will be helpful.

\vspace{0.3in}

\noindent\textbf{Acknowledgement:}
We are grateful to the reviewer for a meticulous reading and for all the corrections and suggestions that have significantly improved the results, presentation, and clarity of this paper. We are especially thankful to the reviewer for improving the results of Remark \ref{remark odometer maps}, Theorem \ref{Thm: Nica covariant characterization}, Theorem \ref{Thm: scalar Nica covariant characterization}, the examples in Section \ref{sec examples}, and Proposition \ref{prop: spect}. First and second named authors thankfull to Harish Chandra Research Institute, Prayagraj, for providing warm hospitality to carry out this initial work. The research of the second named author is supported in part by PM ECRG (ANRF/ECRG/2024/002458/PMS) by ANRF, Government of India. The research of the third named author is supported in part by TARE (TAR/2022/000063) by SERB, Department of Science \& Technology (DST), Government of India.


\begin{thebibliography}{100}

\bibitem{Brin}
M. Brin, {\em On the Zappa–Sz\'{e}p product}, Comm. Algebra. 33 (2005), 393–424.

\bibitem{Brownlowe}
N. Brownlowe, J. Ramagge, D. Robertson, and M. Whittaker, {\em Zappa-Sz\'{e}p products of semigroups and their $C^*$-algebras}, J. Funct. Anal. 266 (2014), 3937–3967.

\bibitem{Bunce}
J. Bunce, {\em Models for $n$-tuples of noncommuting operators}, J. Funct. Anal. 57 (1984), 21–30.
	\bibitem{Clark}
L. Clark, A. an Huef, and I. Raeburn, {\em Phase transitions of the Toeplitz algebras of Baumslag-Solitar semigroups}, Indiana Univ. Math. J. 65 (2016), 2137–2173.


\bibitem{DP}
K.R. Davidson and D.R. Pitts, {\em Invariant subspaces and hyper-reflexivity for free semigroup algebras},
Proc. London Math. Soc. 78 (1999), no. 2, 401–430.

\bibitem{Delgado}
A. Delgado, D. Robinson, J. Derek, and M. Timm, {\em Cyclic normal subgroups of generalized Baumslag-Solitar groups}, Comm. Algebra 45 (2017), 1808–1818.

\bibitem{Douglas}
R. Douglas, {\em On majorization, factorization, and range inclusion of operators on Hilbert space}. Proc. Amer. Math. Soc. 17 (1966), 413-415.


\bibitem{DN}
E. Durszt and B. Sz.-Nagy, {\em Remark to a paper: ``Models for noncommuting operators'' [J. Funct. Anal. 48 (1982), no. 1, 1–11; MR 84h:47010] by A. E. Frazho}, J. Funct. Anal. 52 (1983), 146–147.

\bibitem{Fima}
P. Fima, {\em A note on the von Neumann algebra of a Baumslag-Solitar group}, C. R. Math. Acad. Sci. Paris 349 (2011), 25–27.

\bibitem{Frazho 84}
A. Frazho, {\em Complements to models for noncommuting operators}, J. Funct. Anal. 59 (1984), 445–461.

\bibitem{Frazho}
A. Frazho, {\em Models for noncommuting operators}, J. Funct. Anal. 48 (1982), 1–11.

\bibitem{GT}
V. Gebhardt and S. Tawn, {\em Zappa-Sz\'{e}p products of Garside monoids}, Math. Z. 282 (2016), 341–369.

\bibitem{Boyu Li22}
B. Li, {\em Wold decomposition on odometer semigroups}. Proc. Roy. Soc. Edinburgh Sect. A 152 (2022), 738–755.

\bibitem{Li Yang}
H. Li and D. Yang, {\em Boundary quotient $C^*$-algebras of products of odometers}, Canad. J. Math. 71 (2019), 183–212.

\bibitem{Li}
X. Li, {\em Semigroup $C^*$-algebras and amenability of semigroups}, J. Funct. Anal. 262 (2012), 4302-4340.

\bibitem{Nica}
A. Nica, {\em $C^*$-algebras generated by isometries and Wiener–Hopf operators}, J. Operator Theory, 27 (1992), 17-52.


\bibitem{Popescu 95}
G. Popescu, {\em Multi-analytic operators on Fock spaces}, Math. Ann. 303 (1995), 31–46.

		
\bibitem{Popescue char}
G. Popescu, {\em Characteristic functions for infinite sequences of noncommuting operators}, J. Operator Theory 22 (1989), 51–71.

\bibitem{Popescue}
G. Popescu, {\em Isometric dilations for infinite sequences of noncommuting operators}, Trans. Amer. Math. Soc. 316 (1989), 523–536.

\bibitem{Sarason}
D. Sarason, {\em Generalized interpolation in $H^\infty$}, Trans. Amer. Math. Soc. 127 (1967), 179–203.


\bibitem{JS-SB}
J. Sarkar, {\em Operator theory on symmetrized bidisc}, Indiana Univ. Math. J. 64 (2015), 847–873.
		
\bibitem{Spiel}
J. Spielberg, {\em $C^*$-algebras for categories of paths associated to the Baumslag-Solitar groups}, J. Lond. Math. Soc. 86 (2012), 728–754.

\bibitem{NagFoi10}
B. Sz.-Nagy and C. Foia\c{s}, {\em Harmonic analysis on operators on Hilbert space}, North-Holland, Amsterdam, 1970.



\end{thebibliography}
\end{document}